\documentclass[a4paper,11pt]{article}
\usepackage{mathrsfs}
\usepackage{latexsym}
\usepackage{amsmath,amssymb}
\usepackage{amsthm}
\usepackage{amsfonts}
\usepackage[usenames]{color}
\usepackage{amssymb}
\usepackage{graphicx}
\usepackage{amsmath}
\usepackage{amsfonts}
\usepackage{amsthm}
\usepackage{mathrsfs}
\usepackage{dsfont}
\usepackage{indentfirst}
\usepackage{enumerate}
\usepackage[square, comma, sort&compress, numbers]{natbib}

\newtheoremstyle{mythm}{1.5ex plus 1ex minus .2ex}{1.5ex plus 1ex
minus .2ex}{\kai}{\parindent}{\song\bfseries}{}{1em}{}
\numberwithin{equation}{section}

\newtheorem{thm}{Theorem}[section]
\newtheorem{lemma}{Lemma}[section]
\newtheorem{proposition}{Proposition}[section]
\newtheorem{remark}{Remark} [section]

\allowdisplaybreaks[4]
\newcommand\leqs{\leqslant}
\newcommand\geqs{\geqslant}
\newcommand\g{\mathfrak{g}}

\title{Weighted Eigenvalue Problem for a Class of Singular/Degenerate $k$-Hessian Equations}
\author{Rongxun He and Genggeng Huang}

\begin{document}
\date{}
\maketitle

\begin{abstract}
In this paper, we study the existence and uniqueness of solutions to the weighted eigenvalue problem for $k$-Hessian equation. To achieve this, we establish the uniform a priori estimates for gradient and second derivatives of solutions to Hessian equation with weight $|x|^{2sk}$ on the right-hand-side. We also prove that the eigenfunction is a minimizer of the corresponding functional among all $k$-admissible functions vanishing on the boundary.
\end{abstract}

\section{Introduction}
In this paper, we study a class of $k$-Hessian equations in the following form
\begin{align}
S_k(D^2u)=|x|^{2sk}g(x,u)\quad\text{in }\Omega,
\label{formu01}
\end{align}
where $0\in\Omega\subset\mathbb{R}^n$ is an open bounded domain. $S_k$ is defined by
\begin{align*}
S_k(D^2u)=\sigma_k(\lambda(D^2u)),
\end{align*}
where $\lambda(D^2u)$ are the eigenvalues of Hessian matrix $D^2u$, and $\sigma_k(\lambda)$ denotes the $k$-th elementary symmetric polynomial given by
\begin{align*}
\sigma_k(\lambda)=\sum_{i_1<\cdots<i_k}\lambda_{i_1}\cdots\lambda_{i_k}.
\end{align*}
The Poisson equation and Monge-Ampère equation fall into the form of \eqref{formu01}, respectively, as $k=1$ and $k=n$. Following \cite{CNS85Dirichlet}, a function $u\in C^2(\Omega)\cap C^0(\overline{\Omega})$ is called $k$-admissible if $\lambda(D^2u)$ lies in $\overline{\Gamma}_k$ for all $x\in\Omega$, where $\Gamma_k$ is the symmetric G\aa rding cone given by
\begin{align*}
\Gamma_k=\{\lambda\in\mathbb{R}^n:\sigma_j(\lambda)>0,\ j=1,\ldots,k\}.
\end{align*}
We denote by $\Phi^k(\Omega)$ the set of all $k$-admissible functions in $\Omega$ and by $\Phi_0^k(\Omega)$ the set of all $k$-admissible functions vanishing on the boundary $\partial\Omega$. In addition, a bounded domain $\Omega$ of class $C^2$ is called strictly $(k-1)$-convex, if the boundary $\partial\Omega$ satisfies
\begin{align*}
(\kappa_1(x),\ldots,\kappa_{n-1}(x),K)\in\Gamma_k
\end{align*}
everywhere with some positive constant $K$, where $\kappa_1(x),\cdots,\kappa_{n-1}(x)$ denote the principal curvatures of $\partial\Omega$ at $x$ with respect to inner normal. Obviously, for the case $k=n$, it is equivalent to the usual (strict) convexity.

We first give some known results related to the $k$-Hessian equation \eqref{formu01} for the case $s=0$. The existence of smooth solutions to the Dirichlet problem
\begin{equation}
\left\{
\begin{array}{ll}
 S_k(D^2u)=g  & \text{in }\Omega,  \\
 u=\varphi  & \text{on }\partial\Omega,
\end{array}
\right.
\label{formunew1}
\end{equation}
was first solved by Caffarelli-Nirenberg-Spruck \cite{CNS85Dirichlet} and as well by Ivochkina \cite{Ivochkina85Dirichlet} for the nondegenerate case $g=g(x)>0$ in $\Omega$, provided that $\Omega$ is strictly $(k-1)$-convex. Their approach was further developed and simplified by Trudinger \cite{Trudinger95Dirichlet} to settle with more general type of equations. The above result was extended to the case $g=g(x,u)$ by Li \cite{Li90existence}, via the Leray-Schauder degree theory. Guan \cite{Guan94subsolution} also proved that the geometric condition on $\Omega$ could be replaced by the more general assumption of existence of a strict subsolution.

The degenerate case $g\geqs0$ for \eqref{formunew1} has been extensively studied as well. In this situation, the central issue is the existence of $C^{1,1}$ solutions, or equivalently for some cases the a priori $C^{1,1}$ estimates of solutions. Ivochkina-Trudinger-Wang \cite{ITW04degenerate} solved the $C^{1,1}$ regularity problem \eqref{formunew1} under the assumption $g^{1/k}\in C^{1,1}$, which gave a PDE's proof of Krylov \cite{Krylov89payoff, Krylov95nonlinear}. For degenerate Monge-Ampère equation, Guan-Trudinger-Wang derived the $C^{1,1}$ estimates for solutions in bounded convex domain when $g$ satisfies $g^{1/(n-1)}\in C^{1,1}$. Very recently, Jiao-Wang \cite{JH24degenerate} proved the $C^{1,1}$ regularity for convex solutions of \eqref{formunew1} if $\Omega$ is uniformly convex and $g^{1/(k-1)}\in C^{1,1}$. For general $k$-admissible solutions and $(k-1)$-convex domains, the corresponding question is not solved until now.

When it turns to the case $s\neq0$, there are few results for the classical solvability of the Dirichlet problem of \eqref{formu01}. We note that $|x|^{2s}$ and $|x|^{2sk/(k-1)}$ are not differentiable at the origin for almost every $s\neq0$, so that we could not apply the above estimates for degenerate Hessian equation. In this paper, we will deal with the equation of the form \eqref{formu01} and establish the uniform a priori regularity results for the solution, see Section 3. Accordingly, we can settle with the weighted eigenvalue problem for Hessian equations. In a recent work \cite{HK25variational}, the first author and Ke further proved the existence of (classical) solutions of the equation \eqref{formu01} with homogeneous boundary data, using the variational theory related to the Hardy-Sobolev type inequality for Hessian integral.

In the following, let's review the results concerning the eigenvalue problem of fully nonlinear equations. Lions \cite{Lions85remarks} first solved the eigenvalue problem for Monge-Ampère equation and obtained the existence and uniqueness results for the eigenfunction $\varphi_1$. Tso \cite{Tso90functional} further discussed the functional $I_n(u)=\int_{\Omega}(-u)\det D^2u\,dx$ and proved that $\varphi_1$ minimizes $I_n(u)/\Vert u\Vert_{L^{n+1}(\Omega)}^{n+1}$ among all convex functions vanishing on the boundary. For general $1\leqs k\leqs n$, the eigenvalue problem for $k$-Hessian equation was studied by Wang \cite{Wang94class}. More precisely, there exists a unique eigenvalue $\lambda_1>0$ so that the Dirichlet problem
\begin{equation}
\left\{
\begin{array}{ll}
 S_k(D^2u)=|\lambda u|^k  & \text{in }\Omega,  \\
 u=0  & \text{on }\partial\Omega,
\end{array}
\right.
\label{formu02}
\end{equation}
admits a unique nontrivial solution $\varphi_1\in C^\infty(\Omega)\cap C^{1,1}(\overline{\Omega})$ with $\lambda=\lambda_1$. Moreover, $\lambda_1$ satisfies the following fundamental property:
\begin{align*}
\lambda_1^k=\inf_{u\in\Phi_0^k(\Omega)}\left\{\int_{\Omega}(-u)S_k(D^2u)dx:\Vert u\Vert_{L^{k+1}(\Omega)}=1\right\}.
\end{align*}
The global $C^\infty$ regularity of $\varphi_1$ for Monge-Ampère equation was established on a smooth, strictly convex domain by Hong-Huang-Wang \cite{HHW11degenerate} for $n=2$ and by Le-Savin \cite{LS17degenerate} for $n\geqs2$, while it still remains open for Hessian equation with the case $1<k<n$.

In the current paper, we will extend these results to the weighted situation. Indeed, we will deal with the eigenvalue problem for Hessian equation with weight $|x|^{2sk}$ on the right-hand-side. Since we suppose the domain $\Omega$ always contains the origin, the weight $|x|^{2sk}$ is singular at the origin if $s<0$ while degenerate if $s>0$. Consequently, it is natural that the behavior of the solution differs between these two cases, especially at the origin. We now state our main result as follows:
\begin{thm}
Let $\Omega$ be a strictly $(k-1)$-convex bounded domain containing the origin with the boundary $\partial\Omega\in C^{3,1}$. Suppose $s>-s_0$ for $s_0=\min(1,n/2k)$, then there exists a unique positive constant $\lambda_1=\lambda_1(n,k,s,\Omega)$, so that the eigenvalue problem
\begin{equation}
\left\{
\begin{array}{ll}
 S_k(D^2u)=\big(|x|^{2s}|\lambda u|\big)^k  & \text{in }\Omega,  \\
 u=0  & \text{on }\partial\Omega,
\end{array}
\right.
\label{formu03}
\end{equation}
admits a negative solution $\varphi_1\in\Upsilon(\Omega)$ with $\lambda=\lambda_1$, which is unique up to scalar multiplication. Here, the function space $\Upsilon(\Omega)$ is given by
\begin{equation}
\left\{
\begin{array}{ll}
C^{\infty}(\Omega\setminus\{0\})\cap C^{1,1}(\overline\Omega) &\text{if}\quad s\in(0,\infty),\\
C^{\infty}(\Omega\setminus\{0\})\cap C^{1,1}(\overline\Omega\setminus\{0\})\cap C^{\alpha}(\overline\Omega) &\text{if}\quad s\in(-s_0,0),
\end{array}
\right.
\label{formu04}
\end{equation}
with some constant $\alpha\in(0,1)$. Furthermore, $\lambda_1$ satisfies
\begin{align}
\lambda_1^k=\inf_{u\in\Phi_0^k(\Omega)}\left\{\int_{\Omega}(-u)S_k(D^2u)dx: \int_{\Omega}|x|^{2sk}|u|^{k+1}dx=1\right\}.
\label{formu05}
\end{align}
\label{thm01}
\end{thm}

Theorem \ref{thm01} is an extension of the eigenvalue problem \eqref{formu02}. Note that $\varphi_1\in\Upsilon(\Omega)$ is viewed as a viscosity solution and a weak solution of the Dirichlet problem \eqref{formu03} when $s<0$; see \cite{TW97measure} and \cite{Urbas90viscosity}. For $k=n$, we can further derive the boundary $C^\infty$ regularity for the eigenfunction $\varphi_1$ if provided $\Omega$ smooth and strictly convex, according to the boundary regularity results in \cite{LS17degenerate}.

To prove the existence result in Theorem \ref{thm01}, we first consider the approximation problem
\begin{equation}
\left\{
\begin{array}{ll}
 S_k(D^2u)=\big[(|x|^{2}+\delta^2)^{s}|\lambda u|\big]^k  & \text{in }\Omega,  \\
 u=0  & \text{on }\partial\Omega,
\end{array}
\right.
\label{formu06}
\end{equation}
where $\delta>0$ is a small constant. The existence of solution $(\lambda_\delta, \varphi_\delta)$ to \eqref{formu06} follows by the standard procedure in \cite{Lions85remarks} and \cite{Wang94class}. Then it suffices to establish the uniform a priori $C^2$ estimates for $\varphi_\delta$ independent of $\delta$. Indeed, by taking $\delta\to0$ and extracting a subsequence, we can deduce that $\lambda_\delta$ converges to a positive constant $\lambda_1$ and $\varphi_\delta$ converges to a nontrivial function $\varphi_1$, which is a solution of the eigenvalue problem \eqref{formu03} with $\lambda=\lambda_1$. In the following, we state the regularity results for more general type of equations.
\begin{thm}
Let $\Omega$ be a strictly $(k-1)$-convex bounded domain containing the origin with the boundary $\partial\Omega\in C^{3,1}$. Let $u\in C^{3,1}(\Omega)\cap C^3({\overline\Omega})$ be a $k$-admissible solution of \eqref{formu01} vanishing on the boundary. Suppose that $g^{1/k}\in C^{1,1}(\overline\Omega\times\mathbb{R})$ is a nonnegative function. Then if $-1<s<0$, it holds for every $\beta>1$,
\begin{align}
\sup_{\Omega}|x||Du(x)|\leqs K,\quad \sup_{\Omega}|x|^{2\beta}|D^2u(x)|\leqs L,
\label{formu07}
\end{align}
where the constants $K,L$ depend only on $n,k,s,\beta,\Omega,\Vert u\Vert_{L^\infty(\Omega)}$ and $\Vert g^{1/k}\Vert_{C^{1,1}}$.

On the other hand, if $s>0$ then it holds
\begin{align}
\sup_{\Omega}|Du(x)|\leqs \hat{K},\quad \sup_{\Omega}|D^2u(x)|\leqs \hat{L},
\label{formu08}
\end{align}
where the constants $\hat{K},\hat{L}$ depend only on $n,k,s,\Omega,\Vert u\Vert_{L^\infty(\Omega)}$ and $\Vert g^{1/k}\Vert_{C^{1,1}}$.
\label{thm02}
\end{thm}

Since the origin point is inside the domain, we can derive the gradient and second derivatives estimates on the boundary $\partial\Omega$, using the same argument as in \cite{CNS85Dirichlet} and \cite{Wangnote}. To prove the latter of \eqref{formu07}, we utilize the idea of Pogorelov estimate for Hessian equation, see \cite[Theorem 4.1]{CW01variational}. Indeed, we view the origin as the ``boundary" in the proof of Pogorelov estimate. A similar argument was also discussed by Wang-Zhou \cite{WZ21complex} for complex Hessian equation
\begin{equation*}
\left\{
\begin{array}{ll}
 \sigma_k(u_{i\bar j})=\dfrac{g}{(|z|^2+\delta^2)^\alpha}  & \text{in }B, \\
 u=\varphi  & \text{on }\partial B,
\end{array}
\right.
\end{equation*}
where $0<\alpha<k$, $0<g\in C^{\infty}(B)$ and $\varphi\in C^{1,1}(\partial B)$. They still use the strict positivity of $g$ to yield the estimates, as in \cite{CW01variational}. However, in our version $g$ is supposed to be nonnegative but not strictly positive. Note that the choice $g=|\lambda u|^k$ in \eqref{formu06} fits the above condition. To settle with this difficulty, we make use of the convexity of $|x|^2$ to provide suitable positive terms in the proof, see Theorem \ref{thm12}. Our result applies to the complex equation as well. For the case $s>0$, we additionally derive the estimate \eqref{formu08} by applying the Alexandrov maximum principle to equations of $Du$ and $D^2u$, see Theorem \ref{thm14}.

We finally turn our attention to the uniqueness result of the eigenvalue problem \eqref{formu03}. Since the equation is singular or degenerate at the origin, we fail to apply the classical strong maximum principle directly. Instead, we need to use the basic property of fundamental solution $w_k$ (see \eqref{formunew} for the definition) of $k$-Hessian equation. Precisely, $w_k$ tends to $-\infty$ as $x$ goes to $0$ if $k\leqs n/2$, while it is H\"older continuous of order $2-n/k$ if $k>n/2$. Due to the different behavior at the origin, we treat with these two cases separately.

For the case $k\leqs n/2$, we first prove that the linearized equation of \eqref{formu03}
\begin{equation*}
\left\{
\begin{array}{ll}
 F^{ij}(D^2u)\partial_{ij}v=|x|^{2s}h  & \text{in }\Omega\setminus\{0\},  \\
 v=0  & \text{on }\partial\Omega,
\end{array}
\right.
\end{equation*}
is solvable in $W^{2,p}_{loc}(\Omega\setminus\{0\})\cap C(\overline{\Omega}\setminus\{0\})\cap L^{\infty}(\Omega)$ for some $p>n/2$ if $h\in L^{\infty}(\Omega)$. Then, by applying the comparison principle with $w_k$, we can show that the solution $v$ is unique. Moreover, an application of compact embedding yields that the map $h\mapsto v$ is a compact mapping from $L^\infty(\Omega)$ to itself. Using the spectral theory of linear operators, we then prove the uniqueness of solutions to eigenvalue problem \eqref{formu03}, by a similar argument in \cite[Theorem 4.1]{Wang94class}.

For the case $k>n/2$, the situation is totally different. We need to derive a higher regularity at the origin for the eigenfunction than that for the fundamental solution $w_k$. For $n/2<k<n$, we utilize the Wolff potential theory in \cite{Ld02potential} and the interior gradient estimate in \cite{CW01variational} to deduce the H\"older estimate with an order larger than $2-n/k$, while for $k=n$ we directly apply the interior $C^{1,\gamma}$ estimate for Monge-Ampère equation. Then, using the comparison principle with $w_k$, we reach a contradiction to the a priori regularity estimates of solutions and therefore conclude the uniqueness part of Theorem \ref{thm01}.

This paper is organized as follows. In Section 2, we introduce some preliminary results of the operator $\sigma_k$. In Section 3, we prove the a priori regularity estimates for solutions of Hessian equation with weight. In Section 4 and Section 5, we study the existence and uniqueness results for eigenvalue problem \eqref{formu03}, respectively. Finally, we prove the spectral feature \eqref{formu05} for the eigenvalue $\lambda_1$ in Section 6.

\section{Preliminaries}
Let $S_k$ be defined as above. For $u\in\Phi^k(\Omega)$, we always denote $F(D^2u)=S_k^{1/k}(D^2u)$. For simplicity, we always view the following equations the same:
\begin{align*}
S_k(D^2u)=\psi\qquad\text{and}\qquad F(D^2u)=\psi^{1/k}.
\end{align*}
For latter applications, we denote the notions 
\begin{align*}
S_k^{ij}=\frac{\partial S_k}{\partial u_{ij}},\quad S_k^{ij,pq} =\frac{\partial^2S_k}{\partial u_{ij}\partial u_{pq}} \quad\text{and}\quad
F^{ij}=\frac{\partial F}{\partial u_{ij}},\quad 
F^{ij,pq}=\frac{\partial^2F}{\partial u_{ij}\partial u_{pq}}.
\end{align*}

In the following, we introduce some inequalities for the polynomial $\sigma_k(\lambda)$. For $\lambda\in\Gamma_k$, denote $\sigma_{k;i}(\lambda)=\sigma_k(\lambda) \big|_{\lambda_i=0}$ and $\sigma_{k;ij}(\lambda)= \sigma_k(\lambda)\big|_{\lambda_i=\lambda_j=0}$ for $i\neq j$.
\begin{proposition}\label{pro21}
Assume $\lambda=(\lambda_1,\ldots,\lambda_n)\in\Gamma_k$ with $\lambda_1\geqs\ldots\geqs\lambda_n$. Then it holds
\begin{enumerate}[\rm(i)]
\item $\sum_{i=1}^n\sigma_{k-1;i}(\lambda)=(n-k+1)\sigma_{k-1}(\lambda)$;
\item $\sigma_{k-1;n}(\lambda)\geqs\ldots\geqs\sigma_{k-1;1}(\lambda)>0$;
\item $\lambda_k\geqs0$ and $\sigma_{k-1;k}(\lambda)\geqs\theta\sum_{i=1}^n\sigma_{k-1;i}(\lambda)$ for some $\theta=\theta(n,k)>0$;
\item {\rm(Maclaurin inequality)}
$\big[\binom{n}{k-1}^{-1}\sigma_{k-1}(\lambda)\big]^{1/(k-1)} \geqs\big[\binom{n}{k}^{-1}\sigma_{k}(\lambda)\big]^{1/k}$;
\item {\rm(G\aa rding inequality)}
$\sum_{i=1}^n\mu_i\sigma_{k-1;i}(\lambda)\geqs k[\sigma_k(\lambda)]^{(k-1)/k}[\sigma_k(\mu)]^{1/k}$ for any $\mu\in\Gamma_k$;
\item $\prod_{i=1}^n\sigma_{k-1;i}(\lambda)\geqs C[\sigma_{k}(\lambda)]^{n(k-1)/k}$ for some $C=C(n,k)>0$.
\end{enumerate}
\end{proposition}
For the proof of Proposition \ref{pro21}, we refer to \cite{Garding59inequality, Lieberman96inequality, LT94inequality, Wang94class}. Using the properties of $\sigma_k$, we give some corresponding results for the operator $F=S_k^{1/k}$. Notice that $(\partial/\partial\lambda_i)\sigma_k(\lambda)=\sigma_{k-1;i}(\lambda)$. By (v) we infer that
\begin{align*}
\sum_{i=1}^n\mu_i\frac{\partial}{\partial\lambda_i}\sigma_k^{1/k}(\lambda)\geqs\sigma_k^{1/k}(\mu)\quad\text{holds for any}\quad\lambda,\mu\in\Gamma_k.
\end{align*}
This illustrates that $\sigma_k^{1/k}(\lambda)$ is concave with respect to $\lambda\in\Gamma_k$. Hence, we can deduce that $F(D^2u)=S_k^{1/k}(D^2u)$ is concave with respect to $D^2u$, where $u$ is a $k$-admissible function. By (vi), we obtain that
\begin{align*}
\prod_{i=1}^n\frac{\partial}{\partial\lambda_i}\sigma_k^{1/k}(\lambda)\geqs C(n,k)>0.
\end{align*}
Therefore, it holds that $\det(F^{ij}(D^2u))\geqs C(n,k)>0$, for any $u\in\Phi^k(\Omega)$.

\section{Uniform Estimates}
In this section, we establish the a priori estimates for  gradient and second derivatives of solutions to the following Hessian equation
\begin{align}
S_k(D^2u)=\big[(|x|^2+\delta^2)^sf(x,u)\big]^k,
\label{formu11}
\end{align}
where $f(x,u)$ is a nonnegative function in $\overline\Omega\times\mathbb{R}$ and $s\in(-1,\infty)$ is fixed.

\begin{thm}\label{thm11}
Let $\Omega$ be a strictly $(k-1)$-convex bounded domain containing the origin with the boundary $\partial\Omega\in C^3$. Let $u\in C^3(\Omega)\cap C^1(\overline\Omega)$ be a k-admissible solution of \eqref{formu11} with the boundary condition $\varphi\in C^{1,1}(\partial\Omega)$. Suppose that $f$ is a nonnegative Lipschitz continuous function. Then it holds that
\begin{align}
\sup_{\Omega}|x||Du(x)|\leqs K,
\label{formu12}
\end{align}
where the constant $K$ depends on $n,k,s,\Omega,\Vert u\Vert_{L^\infty(\Omega)},\Vert\varphi\Vert_{C^{1,1}(\partial\Omega)}$ and $\Vert f\Vert_{C^{0,1}}$.
\end{thm}

\begin{proof}
Since $\Omega$ is $(k-1)$-convex, one can construct the supersolution and subsolution near the boundary for \eqref{formu11} (see \cite{CNS85Dirichlet} and \cite{Wangnote}). Indeed, extend $\varphi$ to $\Omega$ such that it is harmonic. By the geometric assumption of $\Omega$, there exists a subsolution $\underline u$ near the boundary such that it vanishes on $\partial\Omega$. Then by the comparison principle, we have
\begin{align*}
\varphi+\sigma\underline u\leqs u\leqs \varphi \quad\text{near  }\partial\Omega,
\end{align*}
provided $\sigma$ large enough, independent of $\delta$. Hence by $\varphi+\sigma\underline u=u=\varphi$ on $\partial\Omega$, we deduce that $\partial_\gamma(\varphi+\sigma\underline u) \leqs \partial_\gamma u \leqs \partial_\gamma\varphi$, where $\gamma$ is the unit inner normal to $\partial\Omega$. Therefore, we obtain the gradient estimate on the boundary.

By \eqref{formu11}, we have 
\begin{align}
F(D^2u)=(|x|^2+\delta^2)^sf(x,u):=\psi(x,u).
\label{formu13}
\end{align}
In order to establish the global estimate \eqref{formu12}, we consider the auxiliary function
\begin{align*}
G(x,\xi)=u_{\xi}(x)\varphi(u)\rho(x),
\end{align*}
where $\rho(x)=|x|,\varphi(u)=1/(M-u)^{1/2}$ and $M=4(\sup_\Omega|u|+1)$. Suppose $G$ attains its maximum at $x=x_0\in\Omega$ and $\xi=e_1$. We also assume $x_0\neq0$, otherwise it is a trivial case. Then, we have $u_1>0$ and $u_i=0$ at $x_0$ for $i\geqs2$. It holds at $x_0$ that $G_i=0$ and $\{G_{ij}\}\leqs0$, which yield that
\begin{align}
u_{1i}=-\frac{u_1}{\varphi\rho}(u_i\varphi'\rho+\varphi\rho_i),
\label{formu14}
\end{align}
and
\begin{align}
0\geqs F^{ij}G_{ij}=&F^{ij}u_{1ij}\varphi\rho+F^{ij}u_1u_{ij}\varphi'\rho+F^{ij}u_1u_iu_j\varphi''\rho+u_1\varphi F^{ij}\rho_{ij}\nonumber\\
&+2u_1\varphi'F^{ij}u_i\rho_j+2F^{ij}u_{1i}(u_j\varphi'\rho+\varphi\rho_j)\nonumber\\
=&\varphi\rho\partial_1\psi+u_1\psi\varphi'\rho+u_1\rho(\varphi''-\frac{2\varphi'^2}{\varphi})F^{ij}u_iu_j+u_1\varphi F^{ij}\rho_{ij}\nonumber\\
&-2u_1\varphi'F^{ij}u_i\rho_j-\frac{2u_1\varphi}{\rho}F^{ij}\rho_i\rho_j.
\label{formu15}
\end{align}
Note that the last equality follows from \eqref{formu14} and formulae $F^{ij}u_{ij}=\psi$ and $F^{ij}u_{1ij}=\partial_1\psi$, the latter of which is obtained by differentiating the equation \eqref{formu13}.

By direct calculation, we have
\begin{align*}
\varphi'=\frac{1}{2(M-u)^{3/2}}\quad\text{and}\quad \varphi''=\frac{3}{4(M-u)^{5/2}}.
\end{align*}
Hence, it holds that $\varphi''-2\varphi'^2/\varphi\geqs1/(16M^{5/2})$. Moreover, we have $|\rho_i|\leqs 1$ and $|\rho_{ij}|\leqs |x|^{-1}=\rho^{-1}$. Denote $\mathcal{F}=\sum F^{ii}$. Then multiplying \eqref{formu15} by $16M^{5/2}$ and observing that $u_1\psi\varphi'\rho\geqs0$, we obtain
\begin{align}
0\geqs-16M^{5/2}\varphi\rho|\partial_1\psi|+\rho F^{11}u_1^3-C\mathcal{F}(\frac{M^2}{\rho}u_1+Mu_1^2),
\label{formu16}
\end{align}
where the constant $C$ depends only on $n,k$. To continue, we assume that $\rho(x_0)|Du(x_0)|>C_1M$, otherwise we are done. Hence by \eqref{formu14}, we have
\begin{align}
u_{11}=-\frac{\varphi'}{\varphi}u_1^2-\frac{\rho_i}{\rho u_1}u_1^2\leqs-\frac{1}{4M}u_1^2<0,
\label{formu17}
\end{align}
provided $C_1$ sufficiently large. Therefore, we deduce that
\begin{align*}
S_{k-1}(D^2u)=S_k^{11}(D^2u)+u_{11}S_{k-1}^{11}(D^2u)-\sum_{i=2}^nu_{1i}^2S_{k-1}^{i1,1i}(D^2u)\leqs S_k^{11}(D^2u).
\end{align*}
By using $\sum S_k^{ii}=(n-k+1)S_{k-1}$, we obtain $F^{11}\geqs\theta\mathcal{F}$ for some $\theta=\theta(n,k)>0$.

To estimate $\mathcal{F}$ from below, we assume that $D^2u$ is diagonal with the new coordinates $y$ by a rotation, and $u_{y_1y_1}\geqs\cdots\geqs u_{y_ny_n}$. Then by \eqref{formu17}, we have
\begin{align}
u_{y_ny_n}\leqs u_{x_1x_1}\leqs-\frac{1}{4M}u_{x_1}^2<0
\label{formu18}
\end{align}
holds at $x_0$. Since $D^2u$ is diagonal,
\begin{align*}
0\leqs S_k(D^2u)=u_{y_ny_n}\sigma_{k-1;n}(\lambda)+\sigma_{k;n}(\lambda)\quad\text{ for }\lambda=\lambda(D^2u).
\end{align*}
By Maclaurin inequality, it follows that
\begin{align*}
0\leqs u_{y_ny_n}\sigma_{k-1;n}(\lambda)+C[\sigma_{k-1;n}(\lambda)]^{k/(k-1)}.
\end{align*}
Therefore, by \eqref{formu18} we obtain
\begin{align*}
\sigma_{k-1;n}(\lambda)\geqs C|u_{y_ny_n}|^{k-1}\geqs C\frac{u_{x_1}^{2k-2}}{M^{k-1}},
\end{align*}
and hence by our assumption $\rho(x_0)u_{x_1}(x_0)>C_1M$ and $s>-1$,
\begin{align}
\mathcal{F}\geqs F^{y_ny_n}=\frac1k \psi^{1-k}\sigma_{k-1;n}\geqs C\frac{\rho^{2k-2}u_{x_1}^{2k-2}}{M^{k-1}}\geqs C'M^{k-1}.
\label{formu19}
\end{align}

Since $f$ is Lipschitz continuous, we have by direct computation
\begin{align*}
|\partial_1\psi(x_0)|\leqs C\big(1+|x_0|^{2s-1})+C\big(1+|x_0|^{2s}\big)u_{x_1}(x_0).
\end{align*}
Then multiplying \eqref{formu16} by $\rho^2\mathcal{F}^{-1}M^{-3}$, we obtain
\begin{align*}
0\geqs-\frac{\widetilde C\rho^3\mathcal{F}^{-1}|\partial_1\psi|}M+\theta \left(\frac{\rho u_1}M\right)^3-C\left[\left(\frac{\rho u_1}M\right)^2+\frac{\rho u_1}M\right].
\end{align*}
Using $s>-1$ and \eqref{formu19}, we conclude that $\rho u_1\leqs K$ at $x_0$. This completes the proof.
\end{proof}

Next, we will establish the uniform estimates for second derivatives.
\begin{thm}\label{thm12}
Let $\Omega$ be a strictly $(k-1)$-convex bounded domain containing the origin with the boundary $\partial\Omega\in C^{3,1}$. Let $u\in C^{3,1}(\Omega)\cap C^{3}(\overline\Omega)$ be a k-admissible solution of \eqref{formu11} vanishing on the boundary. Suppose that $f\in C^{1,1}(\overline\Omega\times\mathbb{R})$ is a nonnegative function. Then it holds that
\begin{align}
\sup_{\Omega}|x|^{2\beta}|D^2u(x)|\leqs L,
\label{formu110}
\end{align}
where $\beta>1$ and the constant $L$ depends on $n,k,s,\beta,\Omega,\Vert u\Vert_{L^\infty(\Omega)},\Vert f\Vert_{C^{1,1}}$ and the constant $K$ in Theorem \ref{thm11}.
\end{thm}

\begin{proof}
By the same argument as in \cite{CNS85Dirichlet} and \cite{Wang94class}, we have
\begin{align*}
\sup_{\partial\Omega}|D^2u|\leqs\tilde L,
\end{align*}
where the constant $\tilde L$ depends only on $n,k,s,f$ and $\Omega$. To establish the global estimate \eqref{formu110}, we first rewrite the equation \eqref{formu11} as
\begin{align*}
F(D^2u)=(|x|^2+\delta^2)^sf(x,u):=\psi(x,u).
\end{align*}
Differentiating this equation with respect to $x_{\g}$, we obtain
\begin{align}
F^{ij}u_{ij\g}=\psi_{\g},\quad F^{ij}u_{ij\g\g}+F^{ij,pq}u_{ij\g}u_{pq\g}=\psi_{\g\g}.
\label{formu111}
\end{align}
Since $f(x,u)\in C^{1,1}(\overline\Omega\times\mathbb{R})$, by direct calculation we obtain
\begin{align}
|\psi_{\g}|&=\Big|(|x|^2+\delta^2)^s\big[f_{x_{\g}}+f_uu_{\g}\big]+2sx_{\g}(|x|^2+\delta^2)^{s-1}f\Big| \nonumber\\
&\leqs C_1(|x|^2+\delta^2)^s|x|^{-1},
\label{formu112}
\end{align}
and
\begin{align}
|\psi_{\g\g}|=&\Big|(|x|^2+\delta^2)^s\big[f_{x_{\g}x_{\g}} +2f_{x_{\g}u}u_{\g}+f_{uu}u_{\g}^2+f_uu_{\g\g}\big]+2s(|x|^2+\delta^2)^{s-1}f \nonumber\\
&+4sx_{\g}(|x|^2+\delta^2)^{s-1}\big[f_{x_{\g}}+f_uu_{\g}\big]
+4s(s-1)x_{\g}^2(|x|^2+\delta^2)^{s-2}f\Big| \nonumber\\
\leqs&C_2(|x|^2+\delta^2)^s|u_{\g\g}|+C_3(|x|^2+\delta^2)^s|x|^{-2},
\label{formu113}
\end{align}
where we utilize the gradient estimate \eqref{formu12} to yield the inequalities.

When $D^2u$ is diagonal at a given point, we have
\begin{equation*}
F^{ij,pq}=\left\{
\begin{array}{ll}
\mu'\sigma_{k-2;ip}+\mu''\sigma_{k-1;i}\sigma_{k-1;p}&\text{if }i=j,p=q,\\
-\mu'\sigma_{k-2;ij}&\text{if }i\neq j,i=q,\text{ and }j=p,\\
0&\text{otherwise,}
\end{array}
\right.
\end{equation*}
where $\mu(t)=t^{1/k}$. Hence, it follows that
\begin{align}
F^{ii}u_{ii\g\g}=\psi_{\g\g}+\sum_{i,j=1}^n \mu'\sigma_{k-2;ij}u_{ij\g}^2-\sum_{i,j=1}^n [\mu'\sigma_{k-2;ij}+\mu''\sigma_{k-1;i}\sigma_{k-1;j}]u_{ii\g}u_{jj\g}.
\label{formu114}
\end{align}

For $\beta>1$, consider the auxiliary function
\begin{align*}
G(x,\xi)=\rho^{\beta}(x)\varphi\Big(\frac12|x|^2|Du|^2\Big)u_{\xi\xi},
\end{align*}
where $\rho(x)=|x|^2$, $\varphi(t)=(1-t/M)^{-a}$ with the constants $M=K^2+1$ and $0<a<1/2$ to be determined later. Suppose $G$ attains its maximum at $x=x_0\in\Omega$ and $\xi=e_1$. We also assume $x_0\neq0$, otherwise it is a trivial case. By a rotation of coordinates, we assume that $D^2u$ is diagonal at $x_0$ with $u_{11}\geqs\cdots\geqs u_{nn}$. It holds at $x_0$ that $(\log G)_i=0$ and $(\log G)_{ii}\leqs0$, which yield that
\begin{align}
0=(\log G)_i=\beta\frac{\rho_i}{\rho}+\frac{\varphi_i}{\varphi}+\frac{u_{11i}}{u_{11}},
\label{formu115}
\end{align}
and
\begin{align}
0\geqs F^{ii}(\log G)_{ii}=\beta F^{ii}\big(\frac{\rho_{ii}}{\rho}-\frac{\rho_i^2}{\rho^2}\big)+F^{ii}\big(\frac{\varphi_{ii}}{\varphi}-\frac{\varphi_i^2}{\varphi^2}\big)+F^{ii}\big(\frac{u_{11ii}}{u_{11}}-\frac{u_{11i}^2}{u_{11}^2}\big).
\label{formu116}
\end{align}

Next, we consider the following two cases separately.
\\
\textbf{Case 1.} $u_{kk}\geqs\varepsilon u_{11}$ for some $\varepsilon>0$.

By \eqref{formu115}, we have
\begin{align}
\frac{u_{11i}}{u_{11}}=-\beta\frac{\rho_i}{\rho}-\frac{\varphi_i}{\varphi}.
\label{formu117}
\end{align}
Hence putting \eqref{formu117} into \eqref{formu116} yields
\begin{align}
0\geqs\beta F^{ii}\big(\frac{\rho_{ii}}{\rho}-(1+2\beta)\frac{\rho_i^2}{\rho^2}\big)+F^{ii}\big(\frac{\varphi_{ii}}{\varphi}-3\frac{\varphi_i^2}{\varphi^2}\big)+F^{ii}\frac{u_{11ii}}{u_{11}}.
\label{formu118}
\end{align}
By the concavity of $F$ and \eqref{formu113}, we have
\begin{align}
F^{ii}\frac{u_{11ii}}{u_{11}}\geqs\frac{\psi_{11}}{u_{11}}\geqs-C_2(|x|^2+\delta^2)^s-\frac{C_3(|x|^2+\delta^2)^s}{|x|^2u_{11}}\geqs-C(|x|^2+\delta^2)^s,
\label{formu119}
\end{align}
provided $|x|^2u_{11}\geqs1$ at $x_0$. Next by direct calculation, we have
\begin{align*}
\varphi'=\frac{a}{M}\Big(1-\frac{t}{M}\Big)^{-a-1}\quad \text{and}\quad \varphi''=\frac{a(a+1)}{M^2}\Big(1-\frac{t}{M}\Big)^{-a-2}.
\end{align*}
Thus for any $r\in(2,+\infty)$, it holds that
\begin{align}
\frac{\varphi''}{\varphi}-r\frac{\varphi'^2}{\varphi^2} =\frac{a(a+1-ra)}{M^2}\Big(1-\frac{t}{M}\Big)^{-2}\geqs0,
\label{formu120}
\end{align}
provided $a<(r-1)^{-1}$. To continue, we compute (note that $D^2u$ is diagonal)
\begin{equation}
\begin{gathered}
\varphi_i=\varphi'\big(|x|^2u_ju_{ij}+x_i|Du|^2\big),\\
\varphi_{ii}=\varphi''\big(|x|^2u_ju_{ij}+x_i|Du|^2\big)^2+\varphi'\Big[|x|^2(u_{ii}^2+u_{\g}u_{ii\g})+4x_iu_iu_{ii}+|Du|^2\Big].
\end{gathered}
\label{formu121}
\end{equation}
Therefore, we have
\begin{align}
F^{ii}\big(\frac{\varphi_{ii}}{\varphi}-r\frac{\varphi_i^2}{\varphi^2}\big)=&\frac{\varphi'}{\varphi}F^{ii}\Big[|x|^2(u_{ii}^2+u_{\g}u_{ii\g})+4x_iu_iu_{ii}+|Du|^2\Big]  \nonumber\\
&\qquad\qquad+\big(\frac{\varphi''}{\varphi}-r\frac{\varphi'^2}{\varphi^2}\big)F^{ii}\big(|x|^2u_ju_{ij}+x_i|Du|^2\big)^2  \nonumber\\
\geqs&\frac{\varphi'}{2\varphi}|x|^2F^{ii}u_{ii}^2+\frac{\varphi'}{\varphi}|x|^2u_{\g}F^{ii}u_{ii\g}-\frac{8\varphi'}{\varphi}F^{ii}u_i^2+\frac{\varphi'}{\varphi}F^{ii}|Du|^2  \nonumber\\
\geqs&\frac{\varphi'}{2\varphi}|x|^2F^{ii}u_{ii}^2+\frac{\varphi'}{\varphi}|x|^2u_{\g}\psi_{\g}-\frac{7\varphi'}{\varphi}F^{ii}u_i^2,
\label{formu122}
\end{align}
if $a$ is given smaller than $(r-1)^{-1}$. Using $\sigma_{k-1;k}\geqs \theta_{n,k}\sum_i\sigma_{k-1;i}$ for some $\theta_{n,k}>0$, we obtain
\begin{align*}
F^{ii}u_{ii}^2>F^{kk}u_{kk}^2\geqs\varepsilon^2\theta_{n,k}\mathcal{F}u_{11}^2,
\end{align*}
where $\mathcal{F}=\sum F^{ii}$. Together with \eqref{formu12} and \eqref{formu112}, we have by \eqref{formu122}
\begin{align*}
F^{ii}\big(\frac{\varphi_{ii}}{\varphi}-r\frac{\varphi_i^2}{\varphi^2}\big)\geqs \tilde\theta|x|^2\mathcal{F}u_{11}^2-\tilde C_1(|x|^2+\delta^2)^s-C_4\mathcal{F}|x|^{-2},
\end{align*}
where $\tilde\theta$ depends on $n,k,\varepsilon,a$ and $M$. Finally, by our choice of $\rho$, we have
\begin{align*}
\beta F^{ii}\big(\frac{\rho_{ii}}{\rho}-(1+2\beta)\frac{\rho_i^2}{\rho^2}\big)\geqs-C_{\beta}\mathcal{F}|x|^{-2}.
\end{align*}
Inserting the above inequalities to \eqref{formu118} with $r=3$, we obtain at $x_0$
\begin{align}
0\geqs\tilde\theta|x|^2\mathcal{F}u_{11}^2-C_{\beta}\mathcal{F}|x|^{-2}-C(|x|^2+\delta^2)^s.
\label{formu123}
\end{align}
Using the arithmetic and geometric mean inequality, we have
\begin{align}
\mathcal{F}=\sum_{i=1}^nF^{ii}\geqs n\left(\prod_{i=1}^nF^{ii}\right)^{\frac1n}\geqs C_{n,k}>0.
\label{formu124}
\end{align}
By multiplying \eqref{formu123} by $\rho^{2\beta-1}\varphi^2\mathcal{F}^{-1}$ and using $s>-1$, we can deduce that $G(x_0)\leqs L$ holds for some constant $L$ depending on $n,k,s,\varepsilon,a,\beta,f,M$ and $\Omega$.
\\[4pt]
\textbf{Case 2.} $u_{kk}<\varepsilon u_{11}$ (and so $|u_{jj}|<\varepsilon u_{11}$ for $j=k,k+1,\ldots,n$).

In this case, we have by \eqref{formu115}
\begin{align}
\frac{\rho_i}{\rho}=-\frac1{\beta}\big(\frac{\varphi_i}{\varphi}+\frac{u_{11i}}{u_{11}}\big),\quad i=2,\ldots,n.
\label{formu125}
\end{align}
Applying \eqref{formu117} for $i=1$ and \eqref{formu125} for $i=2,\ldots,n$ to \eqref{formu116}, we obtain
\begin{align}
0\geqs& \left\{\sum_{i=1}^n\Big[\beta F^{ii}\frac{\rho_{ii}}{\rho}+F^{ii}\big(\frac{\varphi_{ii}}{\varphi}-C_{\nu}\frac{\varphi_i^2}{\varphi^2}\big)\Big]-\beta(1+2\beta)F^{11}\frac{\rho_1^2}{\rho^2}\right\} \nonumber\\
&+\left\{\sum_{i=1}^n F^{ii}\frac{u_{11ii}}{u_{11}}-\big(1+\frac{1+\nu}{\beta}\big)\sum_{i=2}^n F^{ii}\frac{u_{11i}^2}{u_{11}^2}\right\}:=I_1+I_2,
\label{formu126}
\end{align}
where $\nu<\beta-1$. By \eqref{formu120}$\sim$\eqref{formu122}, we can similarly deduce
\begin{align*}
F^{ii}\big(\frac{\varphi_{ii}}{\varphi}-C_{\nu}\frac{\varphi_i^2}{\varphi^2}\big)\geqs\frac{\varphi'}{2\varphi}|x|^2F^{ii}u_{ii}^2-\frac{\varphi'}{\varphi}\tilde C_1(|x|^2+\delta^2)^s-\frac{\varphi'}{\varphi}C_4\mathcal{F}|x|^{-2},
\end{align*}
given $a<(C_{\nu}-1)^{-1}$. Thus by $\rho_{ii}=2$ for $i=1,\ldots,n$, we obtain
\begin{align}
I_1\geqs\frac{2\beta\mathcal{F}}{\rho}+\frac{\varphi'}{2\varphi}|x|^2F^{ii}u_{ii}^2-\frac{\varphi'}{\varphi}\tilde C_1(|x|^2+\delta^2)^s-\frac{\varphi'}{\varphi}C_4\mathcal{F}|x|^{-2}-C_{\beta}F^{11}|x|^{-2}.
\label{formu127}
\end{align}
For $I_2$, we notice that by the concavity of $F$,
\begin{align*}
-\sum_{i,j=1}^n [\mu'\sigma_{k-2;ij}+\mu''\sigma_{k-1;i}\sigma_{k-1;j}]u_{ii1}u_{jj1}=-\sum\frac{\partial^2}{\partial\lambda_i\partial\lambda_j}\mu(S_k(\lambda))u_{ii1}u_{jj1}\geqs0.
\end{align*}
Hence by \eqref{formu114},
\begin{align*}
u_{11}I_2&\geqs \psi_{11}+\sum_{i,j=1}^n\mu'\sigma_{k-2;ij}u_{ij1}^2-\big(1+\frac{1+\nu}{\beta}\big)\sum_{i=2}^n F^{ii}\frac{u_{11i}^2}{u_{11}}  \nonumber\\
&\geqs \psi_{11}+\sum_{i=2}^n\mu'\Big(2\sigma_{k-2;1i}-\big(1+\frac{1+\nu}{\beta}\big)\frac{\sigma_{k-1;i}}{u_{11}}\Big)u_{11i}^2.
\end{align*}
Since $\nu<\beta-1$, by Lemma 3.1 in \cite{CW01variational}, there exists a uniform constant $\varepsilon=\varepsilon(\beta,\nu)>0$ such that if $|\lambda_j|<\varepsilon\lambda_1$ for $j=k,\ldots,n$, then
\begin{align*}
\sigma_{k-2;1i}(\lambda)\geqs\Big(1-\frac{\beta-1-\nu}{2\beta}\Big)\frac{\sigma_{k-1;i}(\lambda)}{\lambda_1}.
\end{align*}
Thus, by \eqref{formu113} we obtain
\begin{align}
I_2\geqs\frac{\psi_{11}}{u_{11}}\geqs-C_2(|x|^2+\delta^2)^s-\frac{C_3(|x|^2+\delta^2)^s}{|x|^2u_{11}}\geqs-C(|x|^2+\delta^2)^s,
\label{formu128}
\end{align}
provided $|x|^2u_{11}\geqs1$ at $x_0$. Combining \eqref{formu126}$\sim$\eqref{formu128}, we have
\begin{align}
0\geqs \frac{\varphi'}{2\varphi}|x|^2F^{11}u_{11}^2+2\beta\mathcal{F}|x|^{-2}-C(|x|^2+\delta^2)^s-\frac{\varphi'}{\varphi}C_4\mathcal{F} |x|^{-2}-C_{\beta}F^{11}|x|^{-2}.
\label{formu129}
\end{align}
Assuming for a moment that at $x_0$ we have
\begin{align}
J:=2\beta\mathcal{F}|x|^{-2}-C(|x|^2+\delta^2)^s-\frac{\varphi'}{\varphi}C_4\mathcal{F} |x|^{-2}\geqs0,
\label{formu130}
\end{align}
then multiplying \eqref{formu129} by $\rho^{2\beta-1}\varphi^2 (F^{11})^{-1}$, we can deduce that $G(x_0)\leqs L$ holds for some constant $L$ depending on $n,k,s,a,\beta,f,M$ and $\Omega$.
\\

Next, we divide two steps to prove \eqref{formu130}. Notice that $\varphi'/\varphi\in[a/M,2a/M]$.
\\
\textbf{Step 1.}
We first assume $\beta\geqs\Theta$ for some $\Theta$ suitably large such that
\begin{align*}
J\geqs2\Theta\mathcal{F}|x|^{-2}-C(|x|^2+\delta^2)^s-\frac{\varphi'}{\varphi}C_4\mathcal{F}|x|^{-2}>0
\end{align*}
holds for every $x\in\Omega$. Moreover, $\Theta$ depends only on $n,k,s,f,K$ and $\Omega$, thanks to \eqref{formu124} and $s>-1$. We then take $0<\nu<\Theta-1$ and $\varepsilon=\varepsilon(\beta,\nu)>0$ fixed. Set $0<a<1/2$ such that $a<(C_\nu-1)^{-1}$ holds for $C_\nu$ in \eqref{formu126}. Hence, combining Case 1 and Case 2, we deduce the estimate \eqref{formu110} $|x|^{2\beta}|D^2u(x)|\leqs L_\beta$ for $\beta\geqs\Theta$.
\\[4pt]
\textbf{Step 2.} For a general $1<\beta<\Theta$, we first take $0<\nu<\beta-1$ and $\varepsilon=\varepsilon(\beta,\nu)>0$ fixed. Then, choose $a>0$ sufficiently small such that $2aC_4/M\leqs\beta$ and $a<(C_\nu-1)^{-1}$ hold for $C_\nu$ in \eqref{formu126}. Next, we select a $r_0>0$ small so that
\begin{align*}
\beta\mathcal{F}|x|^{-2}-C(|x|^2+\delta^2)^s\geqs0 \quad\text{for }x\in B_{r_0},
\end{align*}
uniformly for $\delta>0$ small. Indeed, $r_0=(\beta C_{n,k}/2C)^{1/(2+2s)}$ satisfies the above condition with the constant $C_{n,k}$ in \eqref{formu124}. Therefore, we obtain \eqref{formu130} holds for every $x\in B_{r_0}$.

Repeat the procedure in the proof with the constants $a,\varepsilon,\nu$ given above. Assume $G(x,\xi)$ attains its maximum at a interior point $x_0$. If $x_0\in B_{r_0}$, then both Case 1 and Case 2 hold and hence we can deduce the estimate \eqref{formu110} $|x|^{2\beta}|D^2u(x)|\leqs L_\beta$. Otherwise if $x_0\in B_{r_0}^c$, then it follows at $x_0$
\begin{align*}
|x|^{2\beta}|D^2u(x)|\leqs r_0^{2\beta-2\Theta}|x|^{2\Theta} |D^2u(x)|\leqs r_0^{2\beta-2\Theta}L_{\Theta},
\end{align*}
where the last inequality follows from the case $\beta=\Theta$. Hence, the choice $L_\beta=Cr_0^{2\beta-2\Theta}L_{\Theta}$ satisfies the estimate \eqref{formu110}. This completes the proof.
\end{proof}

\begin{remark}\label{rmk12}
If $f\in C^{1,1}(\overline\Omega\times\mathbb{R})$ and is positive inside $\Omega$, then by Theorem \ref{thm11} and \ref{thm12}, and by the interior regularity theory of nonlinear elliptic equation, we have
\begin{align*}
\Vert u\Vert_{C^{3+\alpha}(\Omega')}\leqs C(\Omega'),
\end{align*}
where $\Omega'\Subset \Omega\setminus\{0\}$ and the constant $C(\Omega')$ is independent of $\delta$.
\end{remark}

To establish the regularity at the origin uniformly for $\delta>0$, we utilize the local H\"older estimates for Hessian equation studied in \cite{Ld02potential,TW97measure} to obtain
\begin{thm}\label{thm13}
Suppose the same conditions as in Theorem \ref{thm11}. Then there exists a constant $\alpha\in(0,1)$ independent of $\delta$ such that
\begin{align}
\Vert u\Vert_{C^{\alpha}(\overline\Omega)}\leqs C,
\label{formu131}
\end{align}
where $C$ depends on $n,k,s,\Omega,\Vert u\Vert_{L^\infty(\Omega)},f$ and the constant $K$ in Theorem \ref{thm11}.
\end{thm}
\begin{proof}
Using the gradient estimate \eqref{formu12}, we deduce the H\"older estimate near the boundary. Thus in order to obtain \eqref{formu131}, we only need to establish the interior H\"older estimate for solutions of \eqref{formu11}. When $k>n/2$, $k$-admissible functions are locally $\alpha$-H\"older continuous with $\alpha=2-n/k$, see \cite{TW97measure}. When $1\leqs k\leqs n/2$, since $s>-1$ we have
\begin{align*}
\left\Vert\big[(|x|^2+\delta^2)^sf(x,u)\big]^k\right\Vert_{L^{p}(\Omega)}\leqs C,
\end{align*}
for some $p>n/2k$ and some constant $C$ independent of $\delta$. Thanks to the potential theory of Hessian equation, we obtain that $u$ is a locally $\alpha$-H\"older continuous function with some $\alpha\in(0,1)$, see \cite[Corollary 3.4]{Ld02potential} and \cite[Corollary 9.1]{Wangnote}. This finishes the proof.
\end{proof}

In the following, we consider the degenerate situation, namely, $s>0$ for the weight $|x|^{2sk}$ in the equation \eqref{formu11}. We here utilize the method in \cite{Wang94class} to deduce the $L^{\infty}$-estimates for gradient and second derivatives of solutions.

\begin{thm}\label{thm14}
Let $\Omega$ be a strictly $(k-1)$-convex bounded domain containing the origin with the boundary $\partial\Omega\in C^{3,1}$. Let $u\in C^{3,1}(\Omega)\cap C^{3}(\overline\Omega)$ be a k-admissible solution of \eqref{formu11} vanishing on the boundary. Suppose that $s>0$ and $f\in C^{1,1}(\overline\Omega\times\mathbb{R})$ is a nonnegative function. Then it holds that
\begin{align}
\sup_{\Omega}|Du|\leqs \hat K,\quad\sup_{\Omega}|D^2u|\leqs \hat L,
\label{formu132}
\end{align}
where the constants $\hat K, \hat L$ depend on $n,k,s,\Omega,\Vert u\Vert_{L^\infty(\Omega)},\Vert f\Vert_{C^{1,1}}$ and the constant $K$ in Theorem \ref{thm11}.
\end{thm}

\begin{proof}
As discussed in Theorem \ref{thm11} and \ref{thm12}, we have the boundary estimates
\begin{align*}
\sup_{\partial\Omega}|Du|\leqs\tilde K,\quad \sup_{\partial\Omega}|D^2u|\leqs\tilde L
\end{align*}
with the constants $\tilde K, \tilde L$ independent of $\delta$. By \eqref{formu11},
\begin{align*}
F(D^2u)=(|x|^2+\delta^2)^s f(x,u):=\psi(x,u).
\end{align*}
Differentiating this equation with respect to $x_{\g}$, we obtain
\begin{align}
F^{ij}u_{ij\g}=\psi_{\g},\quad F^{ij}u_{ij\g\g}+F^{ij,pq}u_{ij\g}u_{pq\g}=\psi_{\g\g}.
\label{formu133}
\end{align}
Since $f\in C^{1,1}(\overline\Omega\times\mathbb{R})$ and $F^{ij}$ is positive-definite in $\Omega$, we have by \eqref{formu133}
\begin{align}
F^{ij}\partial_{ij}(|Du|^2)&=F^{ij}(2u_{\g}u_{ij\g}+2u_{i\g}u_{j\g})\geqs 2u_{\g}\psi_{\g}  \nonumber\\
&=2u_{\g}\Big[(|x|^2+\delta^2)^s (f_{x_{\g}}+f_u u_{\g})+2sx_{\g}(|x|^2+\delta^2)^{s-1}f\Big]  \nonumber\\
&\geqs -C_1(1+|x|^{2s-1})|Du|-C_2|Du|^2.
\label{formu134}
\end{align}
Noticing that $\det(F^{ij})\geqs C_{n,k}>0$, we then apply the Alexandrov maximum principle to \eqref{formu134} and deduce that
\begin{align*}
\sup_{\Omega}|Du|^2&\leqs\sup_{\partial\Omega}|Du|^2+C\Vert (1+|x|^{2s-1})Du\Vert_{L^n(\Omega)}+C\Vert Du\Vert_{L^{2n}(\Omega)}^2\\
&\leqs\sup_{\partial\Omega}|Du|^2+C'\sup_{\Omega}|Du|+C\sup_{\Omega}|Du|^{\frac{n+1}{n}}\Big[\int_\Omega|Du|^{n-1}dx\Big]^{\frac1n}\\
&\leqs\sup_{\partial\Omega}|Du|^2+C'\sup_{\Omega}|Du|+C'\sup_{\Omega}|Du|^{\frac{n+1}{n}}K^{\frac{n-1}{n}},
\end{align*}
where the last inequality follows from the estimate \eqref{formu12}. This implies
\begin{align}
\sup_{\Omega}|Du|^2\leqs C(1+\sup_{\partial\Omega}|Du|^2)\leqs C(1+\tilde K^2).
\label{formu135}
\end{align}
Next using $f\in C^{1,1}(\overline\Omega\times\mathbb{R})$ and the concavity of $F$, we have by \eqref{formu133}
\begin{align}
F^{ij}\partial_{ij}(\Delta u)=&\Delta\psi-\sum_{\g}F^{ij,pq}u_{ij\g}u_{pq\g}\geqs \Delta\psi  \nonumber\\
=&(|x|^2+\delta^2)^s\sum_{\g}\Big[f_{x_{\g}x_{\g}}+2f_{x_{\g}u}u_{\g}+f_{uu}u_{\g}^2+f_uu_{\g\g}\Big]  \nonumber\\
&+(|x|^2+\delta^2)^{s-1}\sum_{\g}4sx_{\g}(f_{x_{\g}}+f_uu_{\g})  \nonumber\\
&+(|x|^2+\delta^2)^{s-2}\sum_{\g}\Big[2s(|x|^2+\delta^2)+4s(s-1)x_{\g}^2\Big]f  \nonumber\\
\geqs& -C_1(1+|x|^{2s-1})-C_2\Delta u,
\label{formu136}
\end{align}
where the last inequality follows by \eqref{formu135} and $2sn+4s(s-1)>0$. Similarly by applying the Alexandrov maximum principle to \eqref{formu136}, we obtain (note that $\Delta u\geqs0$)
\begin{align*}
\sup_{\Omega}\Delta u&\leqs \sup_{\partial\Omega}\Delta u+C\Vert1+|x|^{2s-1}\Vert_{L^n(\Omega)}+C\Vert\Delta u\Vert_{L^n(\Omega)}\\
&\leqs \sup_{\partial\Omega}\Delta u+C'+C\sup_{\Omega}(\Delta u)^{\frac{n-1}{n}}\Big[\int_\Omega\Delta udx\Big]^{\frac1n}\\
&\leqs \sup_{\partial\Omega}\Delta u+C'+C'\sup_{\Omega}(\Delta u)^{\frac{n-1}{n}},
\end{align*}
where we utilize the divergence theorem at the last inequality. This implies
\begin{align*}
\sup_{\Omega}\Delta u\leqs C(1+\sup_{\partial\Omega}\Delta u)\leqs C(1+\tilde L).
\end{align*}
Since $\lambda(D^2u)\in\Gamma_k\subset\Gamma_2$, we have
\begin{align*}
0\leqs2\sigma_2(\lambda(D^2u))=2\sum_{1\leqs i<j\leqs n}(u_{ii}u_{jj}-u_{ij}^2)=(\Delta u)^2-\sum_{i}u_{ii}^2-\sum_{i\neq j}u_{ij}^2,
\end{align*}
which yields that
\begin{align*}
\sup_{\Omega}|D^2u|\leqs\sup_{\Omega}\Delta u\leqs C(1+\tilde L).
\end{align*}
We finally conclude \eqref{formu132} by setting $\hat K=C(1+\tilde K), \hat L=C(1+\tilde L)$. 
\end{proof}

\section{Existence Results}
In this section, we study the existence results for the following weighted eigenvalue problem
\begin{equation}
\left\{
\begin{array}{ll}
S_k(D^2u)=\big(|x|^{2s}|\lambda u|\big)^k &\text{in }\Omega,\\
u=0 &\text{on }\partial\Omega.
\end{array}
\right.
\label{formu31}
\end{equation}
When $s=0$, this problem was studied by Wang in \cite{Wang94class}. Here we consider the general case $s>-s_0$ for some $s_0>0$.

To continue, we first introduce the regularity and existence results for the Dirichlet problem
\begin{equation}
\left\{
\begin{array}{ll}
S_k(D^2u)=[\psi(x,u)]^k &\text{in }\Omega,\\
u=0 &\text{on }\partial\Omega.
\end{array}
\right.
\label{formu32}
\end{equation}
\begin{thm}\label{thm31}
Suppose that $\psi(x,u)\in C^{1,1}(\overline\Omega\times\mathbb{R})$ and
\begin{align*}
\psi(x,u)>0\quad\text{for}\quad u<0.
\end{align*}
\begin{enumerate}[\rm(i)]
\item If $u\in C^{3,1}(\Omega)\cap C^3(\overline\Omega)$ is a negative solution of \eqref{formu32}, we have the estimates
\begin{align*}
\Vert u\Vert_{C^{1,1}(\overline\Omega)}\leqs C\quad\text{and}\quad\Vert u\Vert_{C^{3,\alpha}(\Omega')}\leqs C(\Omega')\quad\text{for any}\quad\Omega'\Subset\Omega,
\end{align*}
with the constants depending on $n,k,\Omega,\psi$ and $\Vert u\Vert_{L^\infty(\Omega)}$.
\item If there exist a subsolution $w$ and a supersolution $v$ of \eqref{formu32} satisfying $w\leqs v$ in $\Omega$, $w\leqs0$ and $v=0$ on $\partial\Omega$, then \eqref{formu32} admits a solution $u\in C^{3,\alpha}(\Omega)\cap C^{1,1}(\overline\Omega)$ with $w\leqs u\leqs v$.
\end{enumerate}
In addition, if $\psi(x,u)$ is strictly positive, then $C^{3,\alpha}$-estimate is up to the boundary.
\end{thm}

A function $u\in\Phi^k(\Omega)$ is said to be a subsolution (or supersolution) of \eqref{formu32} if
\begin{equation*}
\left\{
\begin{array}{ll}
S_k(D^2u)\geqs (\text{or}\leqs)\;[\psi(x,u)]^k &\text{in }\Omega,\\
u\leqs (\text{or}\geqs)\;0 &\text{on }\partial\Omega.
\end{array}
\right.
\end{equation*}
Theorem \ref{thm31} is given in \cite{Wang94class}. The existence result follows by the method of subsolution and supersolution. The procedure is standard and more details are available in \cite{Aubin81subsuper}.

The main result of this section is as follows.
\begin{thm}\label{thm32}
Consider \eqref{formu31} for $s>-s_0$, where $s_0=\min(1,n/2k)$. Then there exists a positive constant $\lambda_1$ depending only on $n,k,s$ and $\Omega$, such that
\begin{enumerate}[\rm(i)]
\item \eqref{formu31} admits a negative solution $\varphi_1\in\Upsilon(\Omega)$ with $\lambda=\lambda_1$,
\item if $\Omega_1\subset\Omega_2$, then $\lambda_1(\Omega_1)\geqs\lambda_1(\Omega_2)$.
\end{enumerate}
Here, the function space $\Upsilon(\Omega)$ is given by \eqref{formu04}:
\begin{equation*}
\left\{
\begin{array}{ll}
C^{\infty}(\Omega\setminus\{0\})\cap C^{1,1}(\overline\Omega) &\text{if}\quad s\in(0,\infty),\\
C^{\infty}(\Omega\setminus\{0\})\cap C^{1,1}(\overline\Omega\setminus\{0\})\cap C^{\alpha}(\overline\Omega) &\text{if}\quad s\in(-s_0,0).
\end{array}
\right.
\end{equation*}
\end{thm}

\begin{proof}
We divide the proof into three steps.
\\[0.5em]
\textbf{Step 1.}
Given $0<\delta<1$, we introduce a nonnegative constant
\begin{align}
\lambda_{\delta}=\sup\{\lambda>0: \text{there exists a solution }u_{\lambda,\delta}\in C^2(\overline\Omega)\text{ of } \eqref{formu34}\},
\label{formu33}
\end{align}
where \eqref{formu34} is given by
\begin{align}
S_k(D^2u)=\big[(|x|^2+\delta^2)^s(1-\lambda u)\big]^k\quad\text{in }\Omega,\quad u=0\quad\text{on }\partial\Omega.
\label{formu34}
\end{align}
We first show that $\lambda_{\delta}$ has positive upper and lower bounds uniformly independent of $\delta$. Let $\eta_{\delta}$ be the solution of
\begin{align*}
S_k(D^2u)=(|x|^2+\delta^2)^{sk}\quad\text{in }\Omega,\quad u=0\quad\text{on }\partial\Omega.
\end{align*}
Since $s>-s_0$, by $L^\infty$-estimate (see \cite[Theorem 2.1]{CW01variational}), we can deduce that there exists a uniform constant $C_0>0$ such that $|\eta_{\delta}|\leqs C_0$. Then
\begin{align*}
S_k\big(D^2(2\eta_{\delta})\big)=\big[2(|x|^2+\delta^2)^s\big]^k\geqs \big[(|x|^2+\delta^2)^s(1-2\lambda\eta_{\delta})\big]^k
\end{align*}
for $\lambda\in\big(0,(2\sup_\Omega|\eta_{\delta}|)^{-1}\big)$. Hence, $\eta_{\delta}$ and $2\eta_{\delta}$ are respectively a supersolution and a subsolution of \eqref{formu34}. By Theorem \ref{thm31}(ii), we obtain a solution $u_{\lambda,\delta}\in C^{3+\alpha}(\overline\Omega)$ of \eqref{formu34} for $\lambda\in\big(0,(2\sup_\Omega|\eta_{\delta}|)^{-1}\big)$. This yields a uniform positive lower bound $(2C_0)^{-1}$ for $\lambda_{\delta}$.

To see that $\lambda_{\delta}$ has a uniform upper bound, we just observe that if $(\lambda,u)$ solves \eqref{formu34}, then we have (note that $u\leqs0$)
\begin{align*}
\Delta u\geqs n\left[\binom{n}{k}^{-1} S_k(D^2u)\right]^{1/k}&=C(n,k)(|x|^2+\delta^2)^s(1-\lambda u)>-\lambda C(n,k)(|x|^2+\delta^2)^su\\
&\geqs\left\{
\begin{array}{ll}
-\lambda C(n,k)|x|^{2s}u&\text{if}\quad s>0\\
-\lambda C(n,k)(|x|^2+1)^su&\text{if}\quad s<0
\end{array}
\right..
\end{align*}
Hence, $\lambda$ is less than the first eigenvalue of a linear operator independent of $\delta$. This yields a uniform finite upper bound for $\lambda_{\delta}$.

We then claim that for $\delta_1\leqs\delta_2$, it holds $\lambda_{\delta_1}\geqs\lambda_{\delta_2}$ if $s>0$ and $\lambda_{\delta_1}\leqs\lambda_{\delta_2}$ if $s<0$. We give a proof for the case $s>0$ here; the case $s<0$ is similar. Suppose on the contrary that $\lambda_{\delta_1}<\lambda_{\delta_2}$ for $s>0$. Then for $\lambda\in(\lambda_{\delta_1},\lambda_{\delta_2})$, there has a solution $u_{\lambda,\delta_2}$ of \eqref{formu34}. Thus, $u_{\lambda,\delta_2}$ satisfies
\begin{align*}
S_k(D^2u)=\big[(|x|^2+\delta_2^2)^s(1-\lambda u)\big]^k\geqs\big[(|x|^2+\delta_1^2)^s(1-\lambda u)\big]^k.
\end{align*}
On the other hand, we can choose a solution $u_{\bar\lambda,\delta_1}$ of \eqref{formu34} for some $\bar\lambda<\lambda_{\delta_1}$. By Hopf's Lemma, we have $\theta u_{\bar\lambda,\delta_1}\geqs u_{\lambda,\delta_2}$ when $\theta>0$ is given sufficiently small. Moreover, it follows that
\begin{align*}
S_k\big(D^2(\theta u_{\bar\lambda,\delta_1})\big) =\theta^k\big[(|x|^2+\delta_1^2)^s(1-\bar\lambda u_{\bar\lambda,\delta_1})\big]^k\leqs\big[(|x|^2+\delta_1^2)^s(1-\lambda\theta u_{\bar\lambda,\delta_1})\big]^k.
\end{align*}
Thus, $\theta u_{\bar\lambda,\delta_1}$ and $u_{\lambda,\delta_2}$ are respectively a supersolution and a subsolution of \eqref{formu34}. By Theorem \ref{thm31}(ii) again, we can obtain a solution $(\lambda,u)$ which solves \eqref{formu34} for $\delta=\delta_1$ and $\lambda>\lambda_{\delta_1}$, which leads to a contradiction to the definition of $\lambda_{\delta_1}$. This proves our claim.

We next consider $\Omega_1\subset\Omega_2$. We will show that $\lambda_{\delta}(\Omega_1)\geqs\lambda_{\delta}(\Omega_2)$ by a similar argument. If it is not true, we can select a $\lambda\in(\lambda_{\delta}(\Omega_1),\lambda_{\delta}(\Omega_2))$ and a solution $u_{\lambda,\delta}$ of \eqref{formu34} for $\Omega=\Omega_2$. Hence, $u_{\lambda,\delta}$ is a subsolution of \eqref{formu34} for $\Omega=\Omega_1$. In order to obtain a supersolution, we can select a solution $u_{\bar\lambda,\delta}$ for $\Omega=\Omega_1$ and $\bar\lambda<\lambda_{\delta}(\Omega_1)$. Then we have $\theta u_{\bar\lambda,\delta}\geqs u_{\lambda,\delta}$ for some $\theta>0$ sufficiently small and $\theta u_{\bar\lambda,\delta}$ is a supersolution of \eqref{formu34} vanishing on $\partial\Omega_1$. By Theorem \ref{thm31}(ii), there exists a solution $(\lambda,u)$ of \eqref{formu34} for $\Omega=\Omega_1$ and $\lambda>\lambda_{\delta}(\Omega_1)$, which contradicts the definition of $\lambda_{\delta}(\Omega_1)$. This illustrates $\lambda_{\delta}(\Omega_1)\geqs\lambda_{\delta}(\Omega_2)$ for $\Omega_1\subset\Omega_2$.
\\[0.5em]
\textbf{Step 2.}
In this step, we consider the limit as $\lambda\to\lambda_{\delta}$.

Using Hopf's Lemma and Theorem \ref{thm31}(ii) as in Step 1, we can deduce that for any $\lambda\in[0,\lambda_{\delta})$, there exists a solution $u_{\lambda,\delta}$ of \eqref{formu34}. We claim that $\Vert u_{\lambda,\delta}\Vert_{L^\infty(\Omega)}$ tends to $+\infty$ as $\lambda\to\lambda_{\delta}$. Indeed, if it is not the case, then by Theorem \ref{thm31}(i), we have $\Vert u_{\lambda,\delta}\Vert_{C^{3+\alpha}(\overline\Omega)}\leqs M_{\delta}$ uniformly for $\lambda\in[0,\lambda_{\delta})$. Extract a subsequence of $u_{\lambda,\delta}$ so that it converges to $u_{\delta}^*$ in $C^3(\overline\Omega)$. By taking the limit, it follows that $(\lambda_{\delta},u_{\delta}^*)$ is a solution of \eqref{formu34}. Then we have
\begin{align*}
S_k\big(D^2(2u_{\delta}^*)\big)&=2^k\big[(|x|^2+\delta^2)^s(1-\lambda_{\delta} u_{\delta}^*)\big]^k =\big[(|x|^2+\delta^2)^s(2-2\lambda_{\delta} u_{\delta}^*)\big]^k\\
&\geqs\big[(|x|^2+\delta^2)^s(1-(\lambda_{\delta}+\varepsilon)2u_{\delta}^*)\big]^k,
\end{align*}
if $\varepsilon$ is a positive constant given by $(2\sup_\Omega |u_{\delta}^*|)^{-1}$. Using Theorem \ref{thm31}(ii) again, there exists a solution $u\in C^2(\overline\Omega)$ of \eqref{formu34} for $\lambda=\lambda_{\delta}+\varepsilon$, which yields a contradiction to the definition of $\lambda_{\delta}$. Therefore, we derive that $\Vert u_{\lambda,\delta}\Vert_{L^\infty(\Omega)}\to+\infty$.

Denote $v_{\lambda,\delta}=u_{\lambda,\delta}/\Vert u_{\lambda,\delta}\Vert_{L^\infty(\Omega)}$, then $-1\leqs v_{\lambda,\delta}\leqs0$, and $v_{\lambda,\delta}$ satisfies
\begin{align}
S_k(D^2v)=\big[(|x|^2+\delta^2)^s(\Vert u_{\lambda,\delta} \Vert_{L^\infty(\Omega)}^{-1} -\lambda v)\big]^k\quad\text{in }\Omega,\quad v=0\quad\text{on }\partial\Omega.
\label{formu35}
\end{align}
For $\delta>0$ fixed, we apply Theorem \ref{thm31}(i) to deduce $\Vert v_{\lambda,\delta}\Vert_{C^{1,1}(\overline\Omega)}\leqs M_{\delta}$ uniformly for $\lambda\in[0,\lambda_{\delta})$. Moreover by the normalization, $v_{\lambda,\delta}$ does not converge uniformly to $0$ in $\overline\Omega$. Thus for any $\Omega'\Subset\Omega$, we have $\Vert u_{\lambda,\delta}\Vert_{L^\infty(\Omega)}^{-1} -\lambda v_{\lambda,\delta}\geqs C(\Omega')$ for some positive constant $C(\Omega')$ independent of $\lambda$. By the regularity theory for uniformly elliptic equation, we derive the local $C^{m,\alpha}$ bounds for all $m\geqs1$. Hence by taking a subsequence, $v_{\lambda,\delta}$ converges to some $\varphi_{\delta}\in C^{\infty}(\Omega)\cap C^{1,1}(\overline\Omega)$ and by passing the limit in \eqref{formu35}, $\varphi_{\delta}$ verifies
\begin{align}
S_k(D^2u)=\big[(|x|^2+\delta^2)^s|\lambda_{\delta}u|\big]^k\quad\text{in }\Omega,\quad u=0\quad\text{on }\partial\Omega.
\label{formu36}
\end{align}
\textbf{Step 3.}
In this step, we take the limit as $\delta\to0$.

In Step 1, we derive finite positive upper and lower bounds of $\lambda_{\delta}$ uniformly for $0<\delta<1$. We also obtain the monotonicity property for $\lambda_{\delta}$: given $\delta_1\leqs\delta_2$, it holds $\lambda_{\delta_1}\geqs\lambda_{\delta_2}$ if $s>0$ and $\lambda_{\delta_1}\leqs\lambda_{\delta_2}$ if $s<0$. Hence, there exists a positive constant $\lambda_1\in(0,+\infty)$ such that $\lambda_{\delta}\to\lambda_1$ as $\delta\to0$. Since $\lambda_{\delta}(\Omega_1)\geqs\lambda_{\delta}(\Omega_2)$ for $\Omega_1\subset\Omega_2$, we also obtain $\lambda_1(\Omega_1)\geqs\lambda_1(\Omega_2)$ by taking $\delta\to0$. This finishes the proof of the second part.

In Step 2, we obtain for every $\delta>0$, there exists a $\varphi_{\delta}\in C^{\infty}(\Omega)\cap C^{1,1}(\overline\Omega)$ to be the solution of \eqref{formu36}. Suppose $\Vert\varphi_{\delta}\Vert_{L^{\infty}(\Omega)}=1$ for every $\delta>0$. We next consider $s>0$ and $s<0$ separately. For the case $-s_0<s<0$, we use Theorem \ref{thm11}$\sim$\ref{thm13} to derive the  uniform estimates
\begin{align*}
\sup_\Omega|x||D\varphi_{\delta}|\leqs K,\quad\sup_\Omega|x|^{2\beta}|D^2\varphi_{\delta}|\leqs L \quad\text{and}\quad\Vert\varphi_{\delta}\Vert_{C^\alpha(\overline\Omega)}\leqs C,
\end{align*}
where $\beta>1$, $\alpha\in(0,1)$ and $K,L,C$ are positive constants independent of $\delta$. Since $\varphi_{\delta}$ is negative inside $\Omega$, by Remark \ref{rmk12} and interior Schauder theory, we have the uniform local $C^{m,\alpha}$ bounds for $\varphi_{\delta}$ for every $\Omega'\Subset\Omega\setminus\{0\}$ and $m\geqs1$. Therefore by extracting a subsequence, $\varphi_{\delta}$ converges to some function $\varphi_1\in C^{\infty}(\Omega\setminus\{0\})\cap C^{1,1}(\overline\Omega\setminus\{0\})\cap C^{\alpha}(\overline\Omega)$, which is a solution to \eqref{formu31} with $\lambda=\lambda_1$.

For the case $s>0$, we use Theorem \ref{thm14} to derive the uniform estimates
\begin{align*}
\sup_{\Omega}|D\varphi_{\delta}|\leqs\hat K\quad\text{and}\quad \sup_{\Omega}|D^2\varphi_{\delta}|\leqs\hat L,
\end{align*}
where $\hat K, \hat L$ are positive constants independent of $\delta$. Similarly by Remark \ref{rmk12} and interior Schauder theory, we have the uniform local $C^{m,\alpha}$ bounds for $\varphi_{\delta}$ for every $\Omega'\Subset\Omega\setminus\{0\}$ and $m\geqs1$. Up to a subsequence, $\varphi_{\delta}$ converges to some function $\varphi_1\in C^{\infty}(\Omega\setminus\{0\})\cap C^{1,1}(\overline\Omega)$, which is a solution to \eqref{formu31} with $\lambda=\lambda_1$. Combining the two cases, we actually deduce that \eqref{formu31} admits a negative solution $\varphi_1\in\Upsilon(\Omega)$ with $\lambda=\lambda_1$, where $\Upsilon(\Omega)$ is given by \eqref{formu04}. This completes the proof.
\end{proof}

Theorem \ref{thm32} presents the existence result for the eigenvalue problem \eqref{formu31}. As stated, the solution $\varphi_1\in\Upsilon(\Omega)$ solves the equation in viscosity sense as well as in Hessian measure sense. We will prove the uniqueness result in the next section. 

The features of $\lambda_1$ in Theorem \ref{thm32} coincide the well-known properties of the first eigenvalue of linear elliptic operators of second order. Hence, we will call $\lambda_1$ the (first) eigenvalue of Hessian equation with weight $|x|^{2sk}$ and $\varphi_1$ its corresponding eigenfunction. Another fundamental feature of the eigenvalue $\lambda_1$ is the formula \eqref{formu05}:
\begin{align*}
\lambda_1^k=\inf_{u\in\Phi_0^k(\Omega)}\left\{\int_{\Omega}(-u)S_k(D^2u)dx: \int_{\Omega}|x|^{2sk}|u|^{k+1}dx=1\right\}.
\end{align*}
This result will be proved in the last section.

\section{Uniqueness Results}
In this section, we will prove the uniqueness of solution to the weighted eigenvalue problem
\begin{equation}
\left\{
\begin{array}{ll}
S_k(D^2u)=\big(|x|^{2s}|\lambda u|\big)^k &\text{in }\Omega,\\
u=0 &\text{on }\partial\Omega,
\end{array}
\right.
\label{formu51}
\end{equation}
where $s>-s_0$ for $s_0=\min(1,n/2k)$. In Section 4, we obtain $(\lambda_1,\varphi_1)$ is a solution of \eqref{formu51}. Hence, we need to show that if there exists another solution $(\lambda^*,\varphi^*)\in(0,+\infty)\times\Upsilon(\Omega)$ of \eqref{formu51}, then $\lambda^*=\lambda_1$ and $\varphi^*=c\varphi_1$ for some positive constant $c$.

First, we prove the Hopf's Lemma for linearized equations of \eqref{formu51}. Here we utilize the idea in \cite{Huang19uniqueness}, where the Monge-Ampère case is considered.
\begin{lemma}\label{lem51}
Suppose that $(\lambda,\varphi)$ is a nontrivial solution of \eqref{formu51}. If $v\in C^2(\Omega\setminus\{0\})\cap C(\overline{\Omega})$ satisfying $v\geqs0$ is a solution of the problem
\begin{equation*}
\left\{
\begin{array}{ll}
F^{ij}(D^2\varphi)\partial_{ij}v=h &\text{in }\Omega,\\
v=0 &\text{on }\partial\Omega,
\end{array}
\right.
\end{equation*}
with $h\leqs0$ and $h\not\equiv0$, then there exists a positive constant $\theta$ such that
\begin{align*}
v(x)\geqs\theta\,dist(x,\partial\Omega)\quad\text{near }\partial\Omega.
\end{align*}
\end{lemma}

\begin{proof}
For $\varrho>0$, set $\widetilde\Omega_{\varrho}:=\{x\in\Omega:dist(x,\partial\Omega)\leqs\varrho\}$. Take $\varrho$ small so that the origin does not belong to $\widetilde\Omega_{2\varrho}$. Then using the maximum principle, there has a positive constant $c_{\varrho}>0$ such that
\begin{align*}
v(x)\geqs c_{\varrho}\quad\text{on }\partial\widetilde\Omega_{\varrho}\cap\Omega.
\end{align*}
Let $\eta$ be the solution to the problem
\begin{equation*}
\left\{
\begin{array}{ll}
S_k(D^2\eta)=1 &\text{in }\Omega,\\
\eta=0 &\text{on }\partial\Omega.
\end{array}
\right.
\end{equation*}
Consider $w=-\varepsilon\varphi+\varrho\eta$. By the concavity of $F$ and gradient estimate of $\varphi$, it holds
\begin{align*}
F^{ij}(D^2\varphi)\partial_{ij}w&\geqs-\varepsilon F(D^2\varphi)+\varrho F(D^2\eta)  \\
&=-\varepsilon|x|^{2s}|\lambda\varphi|+\varrho\geqs-C\varepsilon\varrho+\varrho>0\qquad\text{in }\widetilde\Omega_{\varrho},
\end{align*}
provided $\varepsilon>0$ small. Then take $\varrho>0$ sufficiently small such that $w\geqs-\varepsilon\varphi/2$ in $\widetilde\Omega_{\varrho}$. 

Next, consider $\tau>0$ small enough so that
\begin{align*}
\tau w\leqs-\tau\varepsilon\varphi\leqs C\tau\varepsilon\varrho \leqs c_\varrho\quad\text{on }\partial\widetilde \Omega_{\varrho}\cap\Omega.
\end{align*}
Therefore, we have
\begin{equation*}
\left\{
\begin{array}{ll}
F^{ij}(D^2\varphi)\partial_{ij}v\leqs0<F^{ij}(D^2\varphi) \partial_{ij}(\tau w)\quad\text{in } \widetilde\Omega_{\varrho},\\
v=0=\tau w \quad\text{on }\partial\Omega,\quad
v\geqs c_\varrho\geqs\tau w \quad\text{on }\partial\widetilde\Omega_{\varrho}\cap\Omega.
\end{array}
\right.
\end{equation*}
Applying the comparison principle, we obtain $v\geqs\tau w\geqs-\tau\varepsilon\varphi/2$ in $\widetilde\Omega_{\varrho}$. Since $\varphi$ satisfies the Hopf's Lemma, there exists a constant $\theta_0>0$ such that $\varphi\leqs-\theta_0\,dist(x,\partial\Omega)$ near the boundary. We finally obtain the desired result by setting $\theta=\tau\varepsilon\theta_0/2$.
\end{proof}

In the following, we introduce the fundamental solution $w_k$ to $k$-Hessian equation and we refer the readers to \cite{TW99measure} and \cite{Wangnote} for details. Indeed, if we define
\begin{equation}
w_k(x)=\left\{
\begin{array}{ll}
 |x|^{2-n/k}  & \text{if }k>n/2, \\
 \log|x|  & \text{if }k=n/2, x\neq0,  \\
 -|x|^{2-n/k}  & \text{if }k<n/2, x\neq0,
\end{array}
\right.
\label{formunew}
\end{equation}
and $w_k(0)=-\infty$ for $1\leqs k\leqs n/2$, then $w_k\in\Phi^k(\mathbb{R}^n)$ and we can compute the Hessian measure $\mu_k$
\begin{equation*}
\mu_k(w_k)=C(n,k)\delta_0,
\end{equation*}
where $\delta_0$ denotes the Dirac measure at the origin.

Next, we consider two cases separately: $k\leqs n/2$ and $k>n/2$.

\begin{thm}\label{thm51}
Assume $k>n/2$ and $(\lambda,\varphi)$ is a nontrivial solution of \eqref{formu51}. Then there exists a constant $\alpha>2-n/k$, such that $\varphi\in C^{\alpha}$(or $C^{1,\gamma}$ if $\alpha>1$) at the origin.
\end{thm}
\begin{proof}
For the case $k=n$, it is easy to check that the measure $\mu=(|x|^{2s}|\lambda\varphi|)^k$ satisfies the doubling condition $\mu(\omega)\leqs b\mu(2^{-1}\omega)$ for any convex subdomain $\omega\subset\Omega$. Thus, we can directly apply the interior $C^{1,\gamma}$ estimate for Monge-Ampère equation, see \cite[Section 8.4]{Han16}.

For the case $n/2<k<n$, consider $0<r<1$ small so that $B_r(0)\subset\Omega$. Let $\eta_1$ be the solution to
\begin{align*}
S_k(D^2u)=0\quad\text{in }B_r\quad\text{and}\quad u=\varphi\quad\text{on }\partial B_r,
\end{align*}
and let $\eta_2$ be the solution to
\begin{align*}
S_k(D^2u)=S_k(D^2\varphi)\quad\text{in }B_r\quad\text{and}\quad u=0\quad\text{on }\partial B_r.    
\end{align*}
Then by the maximum principle, we have $\eta_1\geqs\varphi\geqs\eta_1+\eta_2$. Hence, for $x\in B_{r/2}(0)$
\begin{align}
\varphi(x)-\varphi(0)\leqs\eta_1(x)-[\eta_1(0)+\eta_2(0)]= [\eta_1(x)-\eta_1(0)]-\eta_2(0).
\label{formu52}
\end{align}
By the interior gradient estimate established in \cite[Theorem 4.1]{Wangnote}, we have
\begin{align}
|\eta_1(x)-\eta_1(0)|\leqs C(\mathop{\text{osc}}\limits_{B_r}\varphi)\frac{|x|}{r}\leqs Cr^{2-n/k}\frac{|x|}{r}=Cr^{1-n/k}|x|,
\label{formu53}
\end{align}
where the second inequality follows from the local H\"older continuity of order $2-n/k$ for $k$-admissible functions, see \cite{TW97measure}.

On the other hand, by applying the potential theory (see (2.19) in \cite{Ld02potential}), we obtain
\begin{align}
-\eta_2(0)&\leqs CW_k^\mu(0,r)=C\int_0^r\left(\frac{\mu(B_t(0))}{t^{n-2k}}\right)^{1/k}\frac{dt}{t},
\label{formu54}
\end{align}
where the measure $\mu$ is given by $S_k(D^2\varphi)=(|x|^{2s}|\lambda\varphi|)^k$. By direct computation, it yields that $-\eta_2(0)\leqs Cr^{\epsilon+2-n/k}$, where $\epsilon$ denotes $(n+2sk)/k>0$.

Take $\tilde{r}=r^{\epsilon+1}/2$. Then combining \eqref{formu52}$\sim$\eqref{formu54}, we deduce
\begin{align*}
\varphi(x)-\varphi(0)\leqs Cr^{\epsilon+2-n/k}\leqs C\tilde{r}^{\alpha} \quad\text{in }B_{\tilde{r}}(0),
\end{align*}
where $\alpha=\frac{\epsilon+2-n/k}{\epsilon+1}>2-n/k$. Similaly we have $\varphi(0)-\varphi(x)\leqs C\tilde{r}^{\alpha}$ holds in $B_{\tilde{r}}(0)$. Hence, $\varphi$ is H\"older continuous at the origin with exponent $\alpha>2-n/k$.
\end{proof}

\noindent\textbf{Proof of Uniqueness Result for $k>n/2$:}

It is enough to verify the uniqueness result for $k>n/2$. Without loss of generality, assume $\lambda^*\geqs\lambda_1$. 
Suppose on the contrary that $\varphi^*$ is not proportional to $\varphi_1$ in $\Omega$. Then we consider
\begin{align*}
t_0=\max\{t>0:-\varphi^*\geqs-t\varphi_1\text{ in }\Omega\}.
\end{align*}
Then by the concavity of $F$, it follows that
\begin{align*}
F^{ij}(D^2\varphi_1)\partial_{ij}(-\varphi^*+t_0\varphi_1)&\leqs-F(D^2\varphi^*)+t_0F(D^2\varphi_1) \\
&=|x|^{2s}(\lambda^*\varphi^*-t_0\lambda_1\varphi_1) \leqs|x|^{2s}\lambda^*(\varphi^*-t_0\varphi_1)\leqs0.
\end{align*}
By Lemma \ref{lem51} and the definition of $t_0$, we infer that there exists an interior point $x_0$ such that $-\varphi^*+t_0\varphi_1=0$ at $x_0$. An application of Hopf's Lemma for uniform elliptic operators yields that $x_0$ must be the origin. Given $r>0$ small, we have $-\varphi^*+t_0\varphi_1\geqs c_0$ on $\partial B_r$ for some positive constant $c_0$. Therefore, we deduce by the concavity of $F$
\begin{equation*}
\left\{
\begin{array}{ll}
F^{ij}(D^2\varphi_1)\partial_{ij}(-\varphi^*+t_0\varphi_1)\leqs0\leqs F^{ij}(D^2\varphi_1)\partial_{ij}(\varepsilon w_k) & \text{in }B_r\setminus\{0\}, \\
-\varphi^*+t_0\varphi_1\geqs c_0\geqs\varepsilon w_k &\text{on }\partial B_r, \\
-\varphi^*+t_0\varphi_1=0=\varepsilon w_k &\text{on }x=0,
\end{array}
\right.
\end{equation*}
provided $\varepsilon>0$ sufficiently small. Using the comparison principle, we obtain $-\varphi^*+t_0\varphi_1\geqs\varepsilon w_k\geqs0$ in $B_r$. Since $-\varphi^*+t_0\varphi_1$ vanishes at the origin, this yields a contradiction to Theorem \ref{thm51}. Hence, we conclude that $\lambda^*=\lambda_1$ and $\varphi^*=c\varphi_1$ for some $c>0$.      \hfill $\square$
\\

When it turns to the case $k\leqs n/2$, we will deal with the problem by spectral theory, as in \cite{Lions85remarks} and \cite{Wang94class}. Following Kuo-Trudinger\cite{KT07maximum}, let $\Gamma_k^*$ be the dual cone of the G\aa rding cone $\Gamma_k$, which is given by
\begin{align*}
\Gamma_k^*=\{\lambda\in\mathbb{R}^n:\lambda\cdot\xi\geqs0\text{ for all }\xi\in\Gamma_k\}.
\end{align*}
It is easy to check $\Gamma_k^*\subset\Gamma_l^*$ for $k\leqs l$. For $\xi\in\Gamma_k$ and $\lambda\in\Gamma_k^*$, denote
\begin{align*}
\rho_k(\xi)=\left\{\frac{\sigma_k(\xi)}{\binom{n}{k}}\right\}^{1/k}\quad\text{and}\quad\rho_k^*(\lambda)=\inf\left\{\frac{\lambda\cdot\xi}{n}:\xi\in\Gamma_k,\rho_k(\xi)\geqs1\right\}.
\end{align*}
We will employ the same notion as above for matrices $A=(a_{ij})$, writing $A\in\Gamma_k(\text{or }\Gamma_k^*)$ if $\lambda(A)\in\Gamma_k(\text{or }\Gamma_k^*)$ and defining $\rho_k(A)=\rho_k(\lambda(A))$, $\rho_k^*(A)=\rho_k^*(\lambda(A))$.

If we consider the linearized operator $F^{ij}(D^2u)$ of Hessian equation, then by G\aa rding inequality, we have for any $(r_{ij})\in\Gamma_k$
\begin{align*}
\sum_{i,j}F^{ij}(D^2u)r_{ij}\geqs S_k^{1/k}(r_{ij})\geqs0.
\end{align*}
Thus, $F^{ij}(D^2u)\in\Gamma_k^*$ and by the definition of $\rho_k^*$, we obtain
\begin{align}
\rho_k^*(F^{ij})&=\inf\left\{\frac{F^{ij}r_{ij}}{n}:r_{ij}\in\Gamma_k, \rho_k(r_{ij})\geqs1\right\}\nonumber\\
&\geqs\inf\left\{\frac1nS_k^{1/k}(r_{ij}):r_{ij}\in\Gamma_k, \rho_k(r_{ij})\geqs1\right\}=\frac1n\binom{n}{k}^{1/k}.
\label{formu55}
\end{align}
In the following, we study the linearized problem of \eqref{formu51}.

\begin{lemma}\label{lem52}
Assume $k\leqs n/2$ and $(\lambda,\varphi)$ is a nontrivial solution of \eqref{formu51}. Then there exists a unique solution $v\in W_{loc}^{2,p}(\Omega\setminus\{0\})\cap C(\overline{\Omega}\setminus\{0\})\cap L^{\infty}(\Omega)$ of the problem
\begin{equation}
\left\{
\begin{array}{ll}
  F^{ij}(D^2\varphi)\partial_{ij}v=|x|^{2s}h   & \text{in }\Omega\setminus\{0\}, \\
   v=0  & \text{on }\partial\Omega,
\end{array}
\right.
\label{formu56}
\end{equation}
where $h\in L^\infty(\Omega)$ and $p>n/2$.
\end{lemma}
\begin{proof}
Let $\{\eta_m\}\subset C_0^\infty(\Omega)$ be a sequence of smooth cut-off functions such that $0\leqs\eta_m\leqs1$ in $\Omega$ and
\begin{align*}
\eta_m=
\left\{
\begin{array}{ll}
 1 & \text{when}\quad dist(x,\partial\Omega\cup\{0\})>1/m, \\
 0 & \text{when}\quad dist(x,\partial\Omega\cup\{0\})<1/2m.
\end{array}
\right.
\end{align*}
This illustrates that each $\eta_m$ vanishes on a neighborhood of the boundary and the origin. Denote $L=F^{ij}(D^2\varphi)\partial_{ij}$ and $A=F^{ij}(D^2\varphi)$. Then, by setting
\begin{align*}
L_m=\eta_mL+(1-\eta_m)\Delta,\qquad A_m=\eta_mA+(1-\eta_m)I_n,
\end{align*}
we actually obtain a sequence of uniform elliptic operators on $\Omega$, whose elliptic constants may rely on $m$. By our choice of $s$, there exists a constant $p>n/2$ such that $|x|^{2s}h\in L^p(\Omega)$. Hence we can obtain a unique solution $v_m\in W^{2,p}(\Omega)\cap C(\overline{\Omega})$ of the problem 
\begin{equation}
\left\{
\begin{array}{ll}
  L_mv=|x|^{2s}h   & \text{in }\Omega, \\
   v=0  & \text{on }\partial\Omega.
\end{array}
\right.
\label{formu57}
\end{equation}
Furthermore, by the concavity of $\rho_k^*$ and \eqref{formu55}, we indeed have $\rho_k^*(A_m)\geqs\frac1n\binom{n}{k}^{1/k}$ for all $x\in\Omega$. Then using the maximum principle established in \cite{KT07maximum}, we derive the $L^\infty$ estimate for $v_m$:
\begin{align*}
\sup_{\Omega}|v_m|\leqs C\left\Vert\frac{|x|^{2s}h}{\rho_k^*(A_m)}\right\Vert_{L^p(\Omega)}\leqs M,
\end{align*}
where $M$ is a uniform constant depending only on $n,k,s,p,h$ and $\Omega$.

Next, we obtain the uniform boundary gradient estimate for $v_m$. Since $\Omega$ is a strictly $(k-1)$-convex bounded domain of class $C^2$, one can construct a sub-barrier $\underline{v}$ and a super-barrier $\overline{v}$ near the boundary $\partial\Omega$. Indeed, denote
\begin{align*}
    \sigma=-d_x+bd_x^2,\quad\text{where}\quad d_x=dist(x,\partial\Omega).
\end{align*}
We refer to \cite{GT83elliptic} for the computation of the first and second derivatives of the distance function. By a rotation of coordinates, it holds that
\begin{align*}
(D^2\sigma)=\text{diag}\left\{\frac{\kappa_1(1-2bd_x)}{1-\kappa_1d_x},\cdots,\frac{\kappa_{n-1}(1-2bd_x)}{1-\kappa_{n-1}d_x},2b\right\}\quad\text{near }\partial\Omega.
\end{align*}
Thus taking $b>0$ large, we have $\sigma\in\Phi^k$ and $S_k(D^2\sigma)\geqs C_b>0$ holds in $\widetilde\Omega_{\varrho}=\{x\in\Omega:d_x\leqs\varrho\}$ for $\varrho>0$ given sufficiently small. Take $\vartheta>0$ large enough such that
\begin{align*}
S_k\big(D^2(\vartheta\sigma)\big)\geqs (|x|^{2s}|h|)^k\quad \text{in }\widetilde\Omega_{\varrho},\qquad\vartheta\sigma\leqs-M\quad\text{on }\partial\widetilde\Omega_{\varrho}\cap\Omega.
\end{align*}
Denote $\underline{v}=\vartheta\sigma$ and $\overline{v}=-\vartheta\sigma$. We will verify that $\underline{v}\leqs v_m\leqs\overline{v}$ in $\widetilde\Omega_{\varrho}$ for all $m$. Indeed, by \cite[Proposition 2.1]{KT07maximum}, it follows that
\begin{align*}
L_m(\underline{v})&=A_m\cdot D^2\underline{v}\geqs n\rho_k^*(A_m)\rho_k(D^2\underline{v})\\
&\geqs S_k^{1/k}(D^2\underline{v})\geqs|x|^{2s}h =L_mv_m\qquad \text{in }\widetilde\Omega_{\varrho}.
\end{align*}
Furthermore, it holds that $\underline{v}\leqs v_m$ on $\partial\widetilde\Omega_{\varrho}$. Hence, by comparison principle we obtain $\underline{v}\leqs v_m$ in $\widetilde\Omega_{\varrho}$. By a similar argument, we have $v_m\leqs\overline{v}$ in $\widetilde\Omega_{\varrho}$. By taking the limit as $x\to\partial\Omega$, we obtain $|\nabla v_m|\leqs\vartheta$ on $\partial\Omega$ uniformly for $m$, and it holds
\begin{align}
|v_m(x)|\leqs \vartheta\,dist(x,\partial\Omega)\quad\text{near }\partial\Omega,\quad\text{uniformly for }m.
\label{formu58}
\end{align}

Using the interior $W^{2,p}$ estimate, it follows that for every $\Omega'\Subset\Omega\setminus\{0\}$,
\begin{align*}
\Vert v_m\Vert_{W^{2,p}(\Omega')}\leqs C,
\end{align*}
with $C=C(\Omega')$ independent of $m$. Then if we consider a sequence of subdomains tending to $\Omega\setminus\{0\}$, we can extract a subsequence of $\{v_m\}$ converging to a $v\in W_{loc}^{2,p}(\Omega\setminus\{0\})\cap L^{\infty}(\Omega)$. Furthermore, \eqref{formu58} implies that $v$ is continuous up to the boundary $\partial\Omega$. By taking the limit in \eqref{formu57}, $v$ is a solution of the linearized problem \eqref{formu56}.

Finally, we prove the uniqueness of solution. Suppose $\tilde v\in W_{loc}^{2,p}(\Omega\setminus\{0\})\cap C(\overline{\Omega}\setminus\{0\})\cap L^{\infty}(\Omega)$ is another solution of \eqref{formu56}. Then $|v-\tilde v|$ is bounded by a constant $K$. Recall the definition \eqref{formunew} of $w_k$, then for any $\varepsilon>0$, set $r>0$ small enough such that $\varepsilon w_k\leqs-K$ on $\partial B_r$. Hence, by the concavity of $F$ we deduce that
\begin{equation*}
\left\{
\begin{array}{ll}
F^{ij}(D^2\varphi)\partial_{ij}(v-\tilde v)=0\leqs F^{ij}(D^2\varphi)\partial_{ij}(\varepsilon(w_k-c_k)) & \text{in }\Omega\setminus B_r, \\
v-\tilde v\geqs -K\geqs\varepsilon(w_k-c_k) &\text{on }\partial B_r, \\
v-\tilde v=0\geqs\varepsilon(w_k-c_k) &\text{on }\partial\Omega,
\end{array}
\right.
\end{equation*}
where $c_k$ equals $\log(diam(\Omega)+1)$ if $k=n/2$ or 0 if $k<n/2$. Then using the comparison principle, we obtain $v-\tilde v\geqs\varepsilon(w_k-c_k)$ in $\Omega\setminus B_r$. For any fixed compact domain $\Omega'\Subset\Omega\setminus\{0\}$, by letting $\varepsilon\to0$ ($r\to0$ as well), we infer that $v-\tilde v\geqs0$ in $\Omega'$. Due to the arbitrariness of $\Omega'$, we obtain $v-\tilde v\geqs0$ in $\Omega\setminus\{0\}$. By a similar argument on $\tilde v-v$, we conclude that $v\equiv\tilde v$ in $\Omega\setminus\{0\}$. This completes the proof.
\end{proof}

\begin{lemma}\label{lem53}
The mapping $h\mapsto v$ in Lemma \ref{lem52} is compact from $L^\infty(\Omega)$ to itself.
\end{lemma}
\begin{proof}
Given a sequence of $\{h_j\}$ satisfying $\Vert h_j\Vert_{ L^\infty(\Omega)}\leqs C$, we obtain a unique solution $v_{j,m}\in W^{2,p}(\Omega)\cap C(\overline{\Omega})$ of \eqref{formu57} and $v_j\in W^{2,p}_{loc}(\Omega\setminus\{0\})\cap C(\overline{\Omega}\setminus\{0\})\cap L^{\infty}(\Omega)$ of \eqref{formu56}, with $h$ replaced by $h_j$. Moreover, the following holds with uniform constants:
\begin{align*}
&\sup_\Omega|v_j|\leqs M,\qquad |v_j(x)|\leqs\vartheta\, dist(x,\partial\Omega)\quad\text{near }\partial\Omega,\\
&\Vert v_j\Vert_{W^{2,p}(\Omega')}\leqs C(\Omega')\quad\text{for every }\Omega'\Subset\Omega\setminus\{0\}.
\end{align*}
For any $\varepsilon>0$, set $\varrho>0$ small such that $|v_j|\leqs\vartheta\varrho\leqs\varepsilon/2$ in $\widetilde\Omega_{\varrho}=\{x\in\Omega:dist(x,\partial\Omega) \leqs\varrho\}$. Denote $\Omega'=\Omega-\widetilde\Omega_{\varrho}-B_{\varrho}(0)$. Since $W^{2,p}\hookrightarrow C^{\beta}$ for some $\beta\in(0,1)$, then by Arzel\`a-Ascoli Theorem, there exists a subsequence of $\{v_j\}$ denoted by the same notion, which converges in $C(\overline{\Omega'})$. Thus, there has a $N\in\mathbb{N}$ such that
\begin{align}
|v_{j_1}-v_{j_2}|\leqs \varepsilon/2\quad\text{in }\overline{\Omega'}, \quad\text{for any }j_1,j_2\geqs N.
\label{formu59}
\end{align}
Next, applying the maximum principle in \cite{KT07maximum} to \eqref{formu57} on $B_{\varrho}=B_{\varrho}(0)$, we obtain
\begin{align*}
|v_{j_1,m}-v_{j_2,m}|\leqs\sup_{\partial B_{\varrho}}|v_{j_1,m}-v_{j_2,m}|+C\left\Vert\frac{|x|^{2s}(h_{j_1}-h_{j_2})}{\rho_k^*(A_m)}\right\Vert_{L^p(B_{\varrho})}\quad\text{in }B_{\varrho},
\end{align*}
holds for any $m$. Taking $m\to\infty$, we have by \eqref{formu59}, for $j_1,j_2\geqs N$
\begin{align*}
|v_{j_1}-v_{j_2}|\leqs\sup_{\partial B_{\varrho}}|v_{j_1}-v_{j_2}|+C\left\Vert\frac{|x|^{2s}(h_{j_1}-h_{j_2})}{\rho_k^*(A)}\right\Vert_{L^p(B_{\varrho})}<\varepsilon\quad\text{in }B_{\varrho}\setminus\{0\},
\end{align*}
given $\varrho>0$ sufficiently small. Combining these results, we have $|v_{j_1}-v_{j_2}|\leqs\varepsilon$ in $\Omega\setminus\{0\}$ for every $j_1,j_2\geqs N$. We thus infer that for every $L^\infty$-bounded sequence $\{h_j\}$, there exists a subsequence of $\{v_j\}$ converging in $L^\infty(\Omega)$. This finishes the proof.
\end{proof}

Suppose $h\in L^{\infty}(\Omega)$ is negative, then maximum principle yields that the solution $v$ of \eqref{formu56} is strictly positive inside $\Omega\setminus\{0\}$. Under this observation and using Lemma \ref{lem51}$\sim$\ref{lem53}, one can prove the following proposition.
\begin{proposition}\label{prop51}
Given the same notions as above, it holds that
\begin{enumerate}[\rm(a)]
\item There exist a positive eigenvalue $\lambda_{\varphi}$ and a positive solution $\phi\in W_{loc}^{2,p}(\Omega\setminus\{0\})\cap C(\overline{\Omega}\setminus\{0\})\cap L^{\infty}(\Omega)$, such that
\begin{equation*}
\left\{
\begin{array}{ll}
F^{ij}(D^2\varphi)\partial_{ij}\phi=-\lambda_{\varphi} |x|^{2s}\phi & \text{in }\Omega\setminus\{0\}, \\
\phi=0  & \text{on }\partial\Omega.
\end{array}
\right.
\end{equation*}
Moreover, $\lambda_{\varphi}$ is unique and $\phi$ is unique up to scalar multiplication.
\item If $v\in W^{2,p}_{loc}(\Omega\setminus\{0\})\cap C(\overline{\Omega}\setminus\{0\})\cap L^{\infty}(\Omega)$ satisfies $v\geqs0$, $v\not\equiv0$ and
\begin{equation*}
\left\{
\begin{array}{ll}
F^{ij}(D^2\varphi)\partial_{ij}v\geqs(or\leqs)-\lambda|x|^{2s}v & \text{in }\Omega\setminus\{0\},\\
v=0 & \text{on }\partial\Omega,
\end{array}
\right.
\end{equation*}
then $\lambda\geqs\lambda_{\varphi}$ $(or\ \lambda\leqs\lambda_{\varphi})$. And if $\lambda=\lambda_{\varphi}$, then $v$ is proportional to $\phi$.
\end{enumerate}
\end{proposition}

We conclude this section by proving the uniqueness result for the case $k\leqs n/2$.
\\[0.5em]
\textbf{Proof of Uniqueness Result for $k\leqs n/2$:}

Assume that $(\lambda_1,\varphi_1)$ and $(\lambda^*,\varphi^*)$ are both nontrivial solutions of \eqref{formu51}. Without loss of generality, we suppose $\lambda^*\geqs\lambda_1$ and $\varphi^*\leqs\varphi_1$. Then by the linearity of $F$, we have
\begin{align*}
F^{ij}(D^2\varphi_1)\partial_{ij}\varphi_1=F(D^2\varphi_1)=-\lambda_1|x|^{2s}\varphi_1\quad\text{in }\Omega\setminus\{0\},
\end{align*}
and $\varphi_1\leqs0$, $\varphi_1\not\equiv0$. According to Proposition \ref{prop51}(a), $\lambda_1=\lambda_{\varphi_1}$ is uniquely determined. On the other hand, by the concavity of $F$ we have
\begin{align*}
F^{ij}(D^2\varphi_1)\partial_{ij}&(\varphi_1-\varphi^*)\leqs F(D^2\varphi_1)-F(D^2\varphi^*)  \\
&=|x|^{2s}(-\lambda_1\varphi_1 +\lambda^*\varphi^*)\leqs-\lambda^*|x|^{2s}(\varphi_1-\varphi^*)\qquad\text{in }\Omega\setminus\{0\}.
\end{align*}
Since $\varphi_1-\varphi^*\geqs0$ in $\Omega\setminus\{0\}$ and $\varphi_1-\varphi^*=0$ on $\partial\Omega$, we apply Proposition \ref{prop51}(b) to deduce $\lambda^*\leqs\lambda_{\varphi_1}=\lambda_1$. Thus, we obtain $\lambda^*=\lambda_1$ and then $\varphi^*=c\varphi_1$ for some $c>0$. We finally complete the proof. \hfill $\square$

\begin{remark}
The procedure of mapping $h\mapsto v$ for the case $k\leqs n/2$ can be extended to a more general class of linear elliptic operators, whose coefficient matrix belongs to the set $V_k=V_k(\Omega)$ given by
\begin{align*}
V_k=\bigg\{A=(a_{ij}):&(a_{ij})=(a_{ji})>0\text{ in }\overline\Omega\setminus\{0\},\\[-10pt]
&a_{ij}\in C(\overline\Omega\setminus\{0\}), A\in\Gamma_k^*,\text{ and }\rho_k^*(A)\geqs\frac1n\binom{n}{k}^{1/k}\bigg\}.
\end{align*}
If we further denote the eigenvalue $\lambda_A$ as in Proposition \ref{prop51}, then using a similar argument of \cite{Le22spectral}, we can prove the spectral characterization
\begin{align*}
\lambda_1=\inf_{A\in V_k}\lambda_A
\end{align*}
for the eigenvalue $\lambda_1$ of the weighted problem \eqref{formu51}.
\end{remark}

\section{Functional Feature}
For $u\in\Phi_0^k(\Omega)$, we consider the functional $I_k$ given by
\begin{align*}
I_k(u)=\int_{\Omega}(-u)S_k(D^2u)dx.
\end{align*}
By integrating by parts, we have
\begin{align*}
I_k(u)=\frac1k\int_{\Omega}u_iu_jS^{ij}_k(D^2u)dx.
\end{align*}
Denote $\Vert u\Vert_{\Phi_0^k(\Omega)}=[I_k(u)]^{1/(k+1)}$. In \cite{Wang94class}, Wang verified that $\Vert u\Vert_{\Phi_0^k(\Omega)}$ is a norm in $\Phi_0^k(\Omega)$ and obtained Sobolev-type inequalities for the functional $I_k(u)$. We also denote 
\begin{align*}
\Vert u\Vert_{L^{p+1}(\Omega;|x|^{2sk})}=\left(\int_{\Omega}|x|^{2sk}|u|^{p+1}dx\right)^{1/(p+1)}.
\end{align*}
In this section, we will prove a weighted embedding for $u\in\Phi_0^k(\Omega)$ and obtain the formula \eqref{formu05}. Before that, we first introduce two lemmas which are proven in \cite{Wang94class}.
\begin{lemma}\label{lem41}
Suppose that $\psi(x,u)\geqs0$ is nonincreasing for $u\leqs0$. Suppose also that $\psi(x,u)\in C^{1,1}(\overline\Omega\times\mathbb{R})$ is strictly concave with respect to $u$. Then there exists at most one nontrivial solution to the Dirichlet problem
\begin{equation*}
\left\{
\begin{array}{ll}
S_k(D^2u)=[\psi(x,u)]^k &\text{in }\Omega,\\
u=0 &\text{on }\partial\Omega.
\end{array}
\right.
\end{equation*}
\end{lemma}

\begin{lemma}\label{lem42}
Let $\Omega$ be a smooth strictly $(k-1)$-convex bounded domain. Denote $Q=\Omega\times(0,\infty)$. Consider the initial boundary value problem
\begin{equation}
\left\{
\begin{array}{ll}
\mu(S_k(D^2u))-u_t=g(x,t,u) &\text{in }Q,\\
u=\phi &\text{on }\partial Q,
\end{array}
\right.
\label{formu41}
\end{equation}
where $\phi\in C^{4,3}(\overline Q)$, $g\in C^2(\overline Q\times\mathbb{R})$ and $\mu(z)=\log z$. Suppose that $\phi(x,0)\in \Phi^k(\Omega)$ satisfies the compatibility condition
\begin{align}
\mu(S_k(D^2\phi))=g(x,t,\phi)\quad\text{on }\partial\Omega\times\{t=0\}.
\label{formu42}
\end{align}
Suppose also that there exist positive constants $C_1,C_2$ such that
\begin{align*}
g(x,t,u)\leqs C_1+C_2|u|\quad\forall(x,t,u)\in\overline Q\times\mathbb{R}.
\end{align*}
Then there exists a k-admissible solution $u\in C^{3+\alpha,1+\alpha/2}(\overline Q)$ of \eqref{formu41} for some $\alpha>0$. Furthermore, if $C_2=0$, $\Vert\phi\Vert_{C^{4,3}(\overline Q)}<\infty$ and $g$ is irrelevant to $t$, then we have the uniform estimate $\Vert u\Vert_{C^{3+\alpha,1+\alpha/2}(\overline Q)}\leqs C$ for some constant $C>0$.
\end{lemma}

The previous lemma includes the a priori estimates and existence results of solutions to parabolic Hessian equation. Indeed, we say a function $u(x,t)$ is $k$-admissible with respect to the equation \eqref{formu41} if for any given $t\geqs0$, $u(\cdot,t)$ is $k$-admissible. The function $\mu(z)=\log z$ can be replaced by any function $\mu$ satisfying $\mu'(z)>0, \mu''(z)<0$ for all $z>0$,
\begin{align}
\mu(z)\to-\infty\text{ as }z\to0^+,\quad\mu(z)\to+\infty\text{ as }z\to+\infty,
\label{formu43}
\end{align}
and $\mu(\sigma_k(\lambda))$ is concave with respect to $\lambda$. We note that condition \eqref{formu43} is to guarantee $\sigma_k(\lambda)>0$ and hence the admissibility keeps at all time. We refer the reader to \cite{Tso90functional,TW98poincare} for more details on this type of nonlinear parabolic equations. 

The main result of this section is as follows.

\begin{thm}\label{thm41}
Consider $s>-s_0$, where $s_0=\min(1,n/2k)$. Then for any $u\in\Phi_0^k(\Omega)$, we have
\begin{align}
\Vert u\Vert_{L^{p+1}(\Omega;|x|^{2sk})}\leqs C\Vert u\Vert_{\Phi_0^k(\Omega)}\qquad\text{for }p\in[0,k],
\label{formu44}
\end{align}
where the constant $C$ depends only on $n,k,s,p$ and $\Omega$. Moreover, we have
\begin{align}
\inf_{u\in\Phi_0^k(\Omega)}\{\Vert u\Vert_{\Phi_0^k(\Omega)}/
\Vert u\Vert_{L^{k+1}(\Omega;|x|^{2sk})}\}=\lambda_1^{k/(k+1)},
\label{formu45}
\end{align}
where $\lambda_1$ is the eigenvalue given in Theorem \ref{thm32}.
\end{thm}

\begin{proof}
We divide the proof into two steps.
\\[0.5em]
\textbf{Step 1.} We prove \eqref{formu44} for $p\in[0,k)$.

For given $p\in[0,k)$, denote $f_M(z)$ a smooth positive function satisfying
\begin{align*}
f_M(z)=
\left\{
\begin{array}{ll}
(1+|z|)^p     &|z|\leqs M,  \\
|z|^{-2}     & |z|\geqs 2M,
\end{array}
\right.
\end{align*}
and $|z|^{-2}\leqs f_M(z)\leqs 2(1+|z|)^p$ for $|z|\in(M,2M)$, where $M>1$ is a constant. Consider the functional
\begin{align}
J_{M,\delta}(u)=\int_\Omega\Big[\frac{(-u)S_k(D^2u)}{k+1}-(|x|^2+\delta^2)^{sk}F_M(u)\Big] dx,
\label{formu46}
\end{align}
where $F_M(u)=\int_0^{|u|}f_M(z)dz$. For every fixed $M>1$, it follows that $F_M(u)$ is bounded, thanks to our choice of $f_M$. Thus, we obtain $J_{M,\delta}(u)$ is bounded from below. Set
\begin{align*}
d_{M,\delta}=\inf\{J_{M,\delta}(u):u\in\Phi_0^k(\Omega)\}.
\end{align*}
To prove \eqref{formu44}, it then suffices to obtain a uniform lower bound of $d_{M,\delta}$ independent of $\delta>0$ and $M>1$, due to an argument by contradiction.

We first claim that $d_{M,\delta}$ is attained at a function $v_{M,\delta}\in\Phi_0^k(\Omega)$. For any $\varepsilon>0$, select a $\phi_{\varepsilon}^*\in\Phi_0^k(\Omega)\cap C^4(\overline\Omega)$ such that $J_{M,\delta}(\phi_{\varepsilon}^*)\leqs d_{M,\delta}+\varepsilon/2$. Let $\phi_{\varepsilon}$ be the solution of
\begin{align*}
S_k(D^2\phi)=(1-\eta)(|x|^2+\delta^2)^{sk}+\eta S_k(D^2\phi_{\varepsilon}^*)\quad\text{in }\Omega,\quad \phi=0\quad\text{on }\partial\Omega,
\end{align*}
where $\eta\in C_0^\infty(\Omega)$ is a cut-off function satisfying $0\leqs\eta\leqs1$ in $\Omega$ and
\begin{align*}
\eta=1\quad\text{in }\Omega_{\varrho}:=\{x\in\Omega:dist(x, \partial\Omega)>\varrho\}.
\end{align*}
Then by the maximum principle, it follows that
\begin{align*}
\sup_{\Omega_{\varrho}}|\phi_{\varepsilon}(x)-\phi_{\varepsilon}^*(x)| \leqs\sup_{\partial\Omega_{\varrho}}|\phi_{\varepsilon}(x)-\phi_{\varepsilon}^*(x)|\leq\sup_{\partial\Omega_{\varrho}} (|\phi_{\varepsilon}(x)| +|\phi_{\varepsilon}^*(x)|)\to0\quad\text{as }\varrho\to0.
\end{align*}
Taking $\varrho>0$ sufficiently small, we obtain
\begin{align}
J_{M,\delta}(\phi_{\varepsilon})\leqs d_{M,\delta}+\varepsilon\quad\text{and}\quad S_k(D^2\phi_{\varepsilon})=(|x|^2+\delta^2)^{sk}\quad\text{on }\partial\Omega.
\label{formu47}
\end{align}

Consider the following parabolic Hessian problem
\begin{equation}
\left\{
\begin{array}{ll}
\log S_k(D^2u)-u_t=\log\psi_{M,\delta}(x,u)\quad\text{in }Q=\Omega\times(0,\infty),\\
u(\cdot,t)=\phi_{\varepsilon}\quad\text{on }\{t=0\},\qquad u=0\quad\text{on }\partial\Omega\times(0,\infty),
\end{array}
\right.
\label{formu48}
\end{equation}
where $\psi_{M,\delta}(x,u)=(|x|^2+\delta^2)^{sk}f_M(u)$. Since \eqref{formu47}, $\phi_{\varepsilon}$ satisfies the compatibility condition \eqref{formu42}. By the definition of $f_M$, we have $\log \psi_{M,\delta}(x,u)\leqs C_{M,\delta}<+\infty$. Then an application of Lemma \ref{lem42} shows that \eqref{formu48} admits a smooth solution $w_{\varepsilon}(x,t)\in C^{3+\alpha,1+\alpha/2}(\overline Q)$ so that
\begin{align}
\Vert w_{\varepsilon}(x,t)\Vert_{C^{3+\alpha,1+\alpha/2}(\overline Q)}\leqs C
\label{formu49}
\end{align}
with some constant $C$ not depending on $t$. We actually obtain a descent gradient flow for the functional $J_{M,\delta}$. Indeed, thanks to the variational structure of $S_k$ (see \cite{Wangnote}),
\begin{align*}
\frac{d}{dt}J_{M,\delta}(w_{\varepsilon}(\cdot,t))&=-\int_\Omega\Big[S_k(D^2w_{\varepsilon})-\psi_{M,\delta}(x,w_{\varepsilon})\Big]\frac{\partial}{\partial t}w_{\varepsilon}(x,t)dx\\
&=-\int_\Omega\Big[S_k(D^2w_{\varepsilon})-\psi_{M,\delta}(x,w_{\varepsilon})\Big]\log\frac{S_k(D^2w_{\varepsilon})}{\psi_{M,\delta}(x,w_{\varepsilon})}dx\leqs0.
\end{align*}
Hence by \eqref{formu47}, we have $d_{M,\delta}\leqs J_{M,\delta}(w_{\varepsilon}(\cdot,t))\leqs d_{M,\delta}+\varepsilon$ and there exists a sequence $t_j\to+\infty$ such that $(d/dt)J_{M,\delta}(w_{\varepsilon}(\cdot,t_j))\to0$. Thus from \eqref{formu49}, by Arzel\`a-Ascoli Theorem we can extract a subsequence of $w_{\varepsilon}(\cdot,t_j)$ so that it converges to a function $\tilde v_{\varepsilon}(x)\in C^3(\overline\Omega)$, which is the solution of
\begin{align}
S_k(D^2u)=\psi_{M,\delta}(x,u)\quad\text{in }\Omega,\quad u=0\quad\text{on }\partial\Omega.
\label{formu410}
\end{align}
Moreover, it satisfies $d_{M,\delta}\leqs J_{M,\delta}(\tilde v_{\varepsilon})\leqs d_{M,\delta}+\varepsilon$. For $\delta>0$ and $M>1$ fixed, the right-hand-side of the equation \eqref{formu410} is bounded from above. Thus by the comparison principle, there has a positive constant $\tilde C_{M,\delta}$ such that $0\geqs\tilde v_{\varepsilon}\geqs-\tilde C_{M,\delta}$ holds for every $\varepsilon>0$. Hence $\psi_{M,\delta}(x,\tilde v_{\varepsilon})\geqs c_{M,\delta}>0$ uniformly for $\varepsilon>0$. Applying Theorem \ref{thm31}(i) to $\tilde v_{\varepsilon}$, we obtain there exists a constant $C$ independent of $\varepsilon$ such that $\Vert\tilde v_{\varepsilon}\Vert_{C^{3+\alpha}(\overline\Omega)}\leqs C$ holds uniformly. Then up to a subsequence, $\tilde v_{\varepsilon}$ converges in $C^3(\overline{\Omega})$ to a function $v_{M,\delta}$, which is a solution of \eqref{formu410} with $J_{M,\delta}(v_{M,\delta})=d_{M,\delta}$.

We next show that $v_{M,\delta}$ is uniformly bounded for $\delta>0$ and $M>1$. Suppose on the contrary that $R_{M,\delta}=\sup_\Omega|v_{M,\delta}|$ tends to $+\infty$ as $\delta\to0$ or $M\to+\infty$. Set $u_{M,\delta}=v_{M,\delta}/R_{M,\delta}$. Then for every $\epsilon>0$, $u_{M,\delta}$ solves
\begin{align*}
S_k(D^2u_{M,\delta})=R_{M,\delta}^{-k}(|x|^2+\delta^2)^{sk} f_M(R_{M,\delta}u_{M,\delta})\leqs\epsilon(|x|^2+\delta^2)^{sk},
\end{align*}
when $R_{M,\delta}$ is given large enough. Since $s>-s_0$, by $L^\infty$-estimate (see \cite[Theorem 2.1]{CW01variational}) we have $\Vert u_{M,\delta}\Vert_{L^\infty(\Omega)}\to0$ as $\epsilon\to0$ uniformly for $\delta>0$, which leads to a contradiction to $\inf_\Omega u_{M,\delta}=-1$. Hence, we obtain $|v_{M,\delta}|\leqs M_0$ uniformly for $\delta>0$ and $M>1$.

Taking $M>M_0$, by the uniqueness result of Lemma \ref{lem41}, we actually derive that $v_{M,\delta}$ is identical to the unique solution $u_{\delta}$ of the problem
\begin{align*}
S_k(D^2u)=(|x|^2+\delta^2)^{sk}(1+|u|)^p\quad\text{in }\Omega,\quad u=0\quad\text{on }\partial\Omega.
\end{align*}
Denote
\begin{gather*}
H_{\delta}(u)=\int_\Omega\Big[\frac{(-u)S_k(D^2u)}{k+1}-\frac{1}{p+1}(|x|^2+\delta^2)^{sk}(1+|u|)^{p+1}\Big] dx,\\
\tilde d_{\delta}=\inf\{H_{\delta}(u): u\in\Phi_0^k(\Omega)\}.
\end{gather*}
Clearly, $d_{M,\delta}\to \tilde d_{\delta}$ as $M\to+\infty$. Therefore, we have
\begin{align*}
-\infty<H_{\delta}(u_\delta)=\lim_{M\to+\infty}J_{M,\delta}(v_{M,\delta})=\lim_{M\to+\infty}d_{M,\delta}=\tilde d_{\delta}.
\end{align*}
Since $|u_{\delta}|\leqs M_0$ uniformly for $\delta>0$, we infer that $\tilde d_{\delta}$ is uniformly bounded from below. Thus, we obtain \eqref{formu44} holds for $p\in[0,k)$.
\\[0.5em]
\textbf{Step 2.} We prove \eqref{formu45}.

For any fixed $\lambda\in(0,\lambda_1)$, we first choose $\delta_0=\delta_0(\lambda)>0$ small enough such that $\lambda<\lambda_\delta$ for $0<\delta<\delta_0$. By a similar argument of Step 1, we have for $p\in[0,k)$, the unique solution $u=u_{\delta,p,\lambda}$ of the problem
\begin{align}
S_k(D^2u)=(|x|^2+\delta^2)^{sk}(1+|\lambda u|)^p\quad\text{in }\Omega,\quad u=0\quad\text{on }\partial\Omega,
\label{formu411}
\end{align}
satisfies
\begin{align}
H_{\delta,p,\lambda}(u_{\delta,p,\lambda})=\tilde d_{\delta,p,\lambda}=\inf\{H_{\delta,p,\lambda}(u): u\in\Phi_0^k(\Omega)\},
\label{formu412}
\end{align}
where the functional $H_{\delta,p,\lambda}$ is given by
\begin{align*}
H_{\delta,p,\lambda}(u)=\int_\Omega\Big[\frac{(-u)S_k(D^2u)}{k+1}-\frac{1}{\lambda(p+1)}(|x|^2+\delta^2)^{sk}(1+|\lambda u|)^{p+1}\Big] dx.
\end{align*}
Since $p<k$ and $\lambda<\lambda_\delta$, by the definition of $\lambda_\delta$ (see \eqref{formu33}), there exists a solution $\varphi_{\delta,\lambda}$ of \eqref{formu34} which is a subsolution of \eqref{formu411}. From Theorem \ref{thm31}(ii) and by the uniqueness result of Lemma \ref{lem41}, it follows that $0>u_{\delta,p,\lambda}\geqs \varphi_{\delta,\lambda}$. Hence by Theorem \ref{thm31}(i), $u_{\delta,p,\lambda}$ is uniformly bounded in $C^{3+\alpha}(\overline\Omega)$ for $p\in[1,k)$. By extracting a subsequence, we obtain $u_{\delta,p,\lambda}$ converges in $C^3(\overline\Omega)$ to a function $u_{\delta,k,\lambda}$ as $p\to k$. Then it is easy to check that $\tilde d_{\delta,p,\lambda}\to\tilde d_{\delta,k,\lambda}$ as $p\to k$. Thus from \eqref{formu412}, we have
\begin{align*}
\inf\{H_{\delta,k,\lambda}(u):u\in\Phi_0^k(\Omega)\}= H_{\delta,k,\lambda}(u_{\delta,k,\lambda})>-\infty.
\end{align*}
This illustrates that for every $0<\delta<\delta_0$, it follows
\begin{align*}
\inf_{u\in\Phi_0^k(\Omega)}\left\{\dfrac{\displaystyle\int_\Omega(-u)S_k(D^2u)dx}{\displaystyle\int_\Omega(|x|^2+\delta^2)^{sk}|u|^{k+1}dx} \right\}\geqs\lambda^k.
\end{align*}
By the arbitrariness of $\lambda<\lambda_1$ and $0<\delta<\delta_0(\lambda)$, we obtain
\begin{align}
\int_\Omega(-u)S_k(D^2u)dx\geqs\lambda_1^k\int_\Omega|x|^{2sk} |u|^{k+1}dx
\label{formu413}
\end{align}
holds for all $u\in\Phi_0^k(\Omega)$. On the other hand, the first eigenfunction $\varphi_1$ of \eqref{formu31} verifies
\begin{align*}
\int_\Omega(-\varphi_1)S_k(D^2\varphi_1)dx= \lambda_1^k\int_\Omega|x|^{2sk} |\varphi_1|^{k+1}dx.
\end{align*}
Combining this with \eqref{formu413}, we finally conclude \eqref{formu45}. This finishes the proof.
\end{proof}

\begin{remark}\label{rmk61}
In the recent work \cite{HK25variational}, the weighted embedding \eqref{formu44} is proven for general $p\in[0,k^*-1]$, where $k^*=k^*(s)$ is the critical exponent given by
\begin{align*}
k^*\left\{
\begin{array}{ll}
 =\frac{(k+1)(n+2sk)}{n-2k} &\text{if }2k<n\text{ and }s\leqs0,\\
 =\frac{(k+1)n}{n-2k} & \text{if }2k<n\text{ and }s>0, \\
 <\infty    & \text{if }2k=n, \\
 =\infty    & \text{if }2k>n.
\end{array}
\right.
\end{align*}
The proof relies on a reduction to radially symmetric functions by means of a descent gradient flow, as in \cite{Wang94class}.
\end{remark}

\small
\bibliographystyle{plain}
\bibliography{cite}

\medskip\medskip
{\em Address and E-mail:}

\medskip\medskip
{\em Rongxun He}

{\em School of Mathematical Sciences, Fudan University}

{\em rxhe24@m.fudan.edu.cn}

\medskip\medskip
{\em Genggeng Huang}

{\em School of Mathematical Sciences, Fudan University}

{\em genggenghuang@fudan.edu.cn}

\end{document}